\documentclass[9pt]{article}
\usepackage{amsfonts,amssymb,amsmath,comment}
\usepackage{graphicx,graphics,color}
 \makeatletter
 \@addtoreset{equation}{section}
 \makeatother

 \newcounter{enunciato}[section]

 \newtheorem{ittheorem}{Theorem}
 \newtheorem{itlemma}{Lemma}
 \newtheorem{itproposition}{Proposition}
 \newtheorem{itdefinition}{Definition}
 \newtheorem{itcorollary}{Corollary}
 \newtheorem{itconjecture}{Conjecture}

 \newenvironment{theorem}{\addtocounter{enunciato}{1}
 \begin{ittheorem}}{\end{ittheorem}}

 \newenvironment{lemma}{\addtocounter{enunciato}{1}
 \begin{itlemma}}{\end{itlemma}}

\newcommand{\halmos}{\rule{1ex}{1.4ex}}

\newenvironment{proof}{\noindent {\em Proof}.\,\,}
   {\hspace*{\fill}$\halmos$\medskip}

\def \ba {\begin{array}}
\def \ea {\end{array}}

\def \Z {{\mathbb Z}}
\def \R {{\mathbb R}}

\def \N {{\mathbb N}}

\def \P {{\mathbb P}}
\def \E {{\mathbb E}}
\def \T {{\mathbb T}}

\def \cT {{\mathcal T}}
\def \cB {{\mathcal B}}
\def \cF {{\mathcal F}}
\def \cE {{\mathcal E}}
\def \cI {{\mathcal I}}
\def \cA {{\mathcal A}}
\def \cB {{\mathcal B}}
\def \cC {{\mathcal C}}

\def \ra {\rightarrow}

\def \cp {\mathrm{cap}\,}
\def\O{\Omega}
\def\eps{\varepsilon}

\begin{document}
\title{Torsional rigidity for regions\\
with a Brownian boundary}

\author{\renewcommand{\thefootnote}{\arabic{footnote}}
M.\ van den Berg
\footnotemark[1]
\\
\renewcommand{\thefootnote}{\arabic{footnote}}
E.\ Bolthausen
\footnotemark[2]
\\
\renewcommand{\thefootnote}{\arabic{footnote}}
F.\ den Hollander
\footnotemark[3]
}

\footnotetext[1]{ School of Mathematics, University of Bristol,
University Walk, Bristol BS8 1TW, United Kingdom. }

\footnotetext[2]{
Institut f\"ur Mathematik, Universit\"at Z\"urich,
Winterthurerstrasse 190, CH-8057 Z\"urich, Switzerland.
}

\footnotetext[3]{
Mathematical Institute, Leiden University, P.O.\ Box 9512,
2300 RA Leiden, The Netherlands.
}
\date{5 July 2017}

\maketitle

\begin{abstract}
Let $\T^m$ be the $m$-dimensional unit torus, $m\in\N$. The torsional rigidity of an open
set $\Omega \subset \T^m$ is the integral with respect to Lebesgue measure over all
starting points $x\in\Omega$ of the expected lifetime in $\Omega$ of a Brownian motion
starting at $x$. In this paper we consider $\Omega = \T^m \backslash \beta[0,t]$, the
complement of the path $\beta[0,t]$ of an independent Brownian motion up to time $t$.
We compute the leading order asymptotic behaviour of the expectation of the torsional
rigidity in the limit as $t \to \infty$. For $m=2$ the main contribution comes from the
components in $\T^2 \backslash \beta [0,t]$ whose inradius is comparable to the largest
inradius, while for $m=3$ most of $\T^3 \backslash \beta [0,t]$ contributes. A similar
result holds for $m \geq 4$ after the Brownian path is replaced by a shrinking Wiener
sausage $W_{r(t)}[0,t]$ of radius $r(t)=o(t^{-1/(m-2)})$, provided the shrinking is slow
enough to ensure that the torsional rigidity tends to zero. Asymptotic properties of the
capacity of $\beta[0,t]$ in $\R^3$ and $W_1[0,t]$ in $\R^m$, $m \geq 4$, play a central
role throughout the paper. Our results contribute to a better understanding of the geometry
of the complement of Brownian motion on $\T^m$, which has received a lot of attention in
the literature in past years.

\vskip 0.5truecm
\noindent
{\it AMS} 2000 {\it subject classifications.} 35J20, 60G50.

\medskip\noindent
{\it Acknowledgment.} The authors acknowledge support by The Leverhulme Trust
through International Network Grant \emph{Laplacians, Random Walks, Bose Gas,
Quantum Spin Systems}. EB is supported by SNSF-grant 20-100536/1. FdH is
supported by ERC Advanced Grant 267356-VARIS and NWO Gravitation Grant
024.002.003-NETWORKS. The authors thank Greg Lawler for helpful discussions
about capacity of the Wiener sausage.

\end{abstract}

%%%%%%%%%%%%%%%%%%%%%%%%%%%%%%%%%%%%%%%

\section{Background, main results and discussion}
\label{results}

Section~\ref{background} provides our motivation for looking at torsional rigidity, and
points to the relevant literature. Section~\ref{torsrig} introduces our main object of
interest, the torsional rigidity of the complement of Brownian motion on the unit torus.
Section~\ref{asympscal} states our main theorems. Section~\ref{disc} places these
theorems in their proper context and makes a link with the principal Dirichlet eigenvalue
of the complement. Section~\ref{sketch} gives a brief sketch of the main ingredients of
the proofs and provides an outline of the rest of the paper.

%%%%

\subsection{Background on torsional rigidity}
\label{background}

Let $(M,g)$ be a geodesically complete, smooth $m$-dimensional Riemannian manifold
without boundary, and let $\Delta$ be the Laplace-Beltrami operator acting in $L^2(M)$. We will in addition assume that $M$ is stochastically complete. That is, Brownian motion on $M$, denoted by
$(\tilde\beta(s),s \geq 0; \tilde\P_x, x\in M)$, with generator $\Delta$ exists for all positive time. The latter is guaranteed if for example the Ricci curvature on $M$ is bounded from below. See \cite{AG} for further details.
For an open, bounded subset  $\Omega\subset M$, and $x\in \Omega$ we define the first exit time of Brownian motion by
\begin{equation}
\label{e1}
\tilde\tau_{\Omega}=\inf\{s \geq 0\colon\, \tilde\beta(s)\notin \Omega\}.
\end{equation}
It is well known that
\begin{equation}\label{e3a}
u_{\Omega}(x;t)=\tilde\P_x[\tilde\tau_{\Omega}>t]
\end{equation}
is the unique solution of
\begin{equation*}%\label{a2}
\frac{\partial u}{\partial t} = \Delta u,\, u(\,\cdot\,;t)\in H_0^1(\O),\,t>0,
\end{equation*}
with initial condition  $u(x;0)=1$.
The requirement $u(\,\cdot\,;t)\in H_0^1(\O),\, t>0,$ represents the Dirichlet boundary condition.
If we denote the expected lifetime of Brownian motion in $\Omega$ by
\begin{equation}
\label{e2}
v_{\Omega}(x) = \tilde\E_x[\tilde\tau_{\Omega}], \qquad x\in \Omega,
\end{equation}
where $\tilde\E_x$ denotes expectation with respect to $\tilde\P_x$, then
\begin{equation}\label{e3b}
v_{\Omega}(x)=\int_0^{\infty}dt\,u_{\Omega}(x;t).
\end{equation}
It is straightforward to verify that
$v_{\Omega}$, the \emph{torsion function} for $\Omega$, is the unique solution of
\begin{equation}
\label{e3}
-\Delta v=1,\, v\in H_0^1(\Omega).
\end{equation}
The \emph{torsional rigidity} of $\Omega$ is the set function defined by
\begin{equation}
\label{e3alt}
\cT(\Omega)=\int_{\Omega} dx\,v_{\Omega}(x).
\end{equation}
The torsional rigidity of a cross section of a cylindrical beam found its origin in the computation of
the angular change when a beam of a given length and a given modulus of rigidity is exposed
to a twisting moment. See for example \cite{TG}.

From a mathematical point of view both the torsion function $v_{\Omega}$ and the torsional rigidity $\cT(\Omega)$ have
been studied by analysts and probabilists. Below we just list a few key results. In analysis, the torsion function is an essential ingredient for the study of gamma-convergence of sequences of sets. See chapter 4 in \cite{BB}. Several isoperimetric inequalities have been obtained for the
torsional rigidity when $M=\R^m$. If $\Omega \subset \R^m$ has finite Lebesgue measure $|\Omega|$, and
$\Omega^*$ is the ball with the same Lebesgue measure, centred at $0$, then $\cT(\Omega)
\leq \cT(\Omega^*)$. The following stability result for torsional rigidity was obtained in
\cite{BP}:
\begin{equation*}
\frac{\cT(\Omega^*)-\cT(\Omega)}{\cT(\Omega^*)} \geq C_m \cA(\Omega)^3.
\end{equation*}
Here, $\cA(\Omega)$ is the Fraenkel asymmetry of $\Omega$, and $C_m$ is an $m$-dependent
constant. The Kohler-Jobin isoperimetric inequality \cite{KJ1},\cite{KJ2} states
that
\begin{equation*}
\lambda_1(\Omega)^{(m+2)/2}\,\cT(\Omega) \geq
\lambda_1(\Omega^*)^{(m+2)/2}\,\cT(\Omega^*).
\end{equation*}
Stability results have also been obtained for the Kohler-Jobin inequality \cite{BP}. A
classical isoperimetric inequality \cite{T} states that
\begin{equation*}
\|v_{\Omega} \|_{L^{\infty}(\Omega)} \leq v_{\Omega^*}(0).
\end{equation*}

In probability, the first exit time moments of Brownian motion have been studied in for example \cite{BvdBC} and \cite{McD}.
These moments are Riemannian invariants, and the $L^1$-norm of the first moment is the
torsional rigidity.

The \emph{heat content} of $\Omega$ at time $t$ is defined as
\begin{equation}
\label{a5}
H_\O(t) = \int_\O u_\O(x;t)\,dx.
\end{equation}
This quantity represents the amount of heat in $\O$ at time $t$, if $\O$ is at initial temperature $1$, while the boundary of $\O$ is at temperature $0$ for all $t>0$.
By \eqref{e3a}, $0\le u_{\O}\le 1$, and so
\begin{equation*}%\label{a8}
0 \leq H_\O(t) \le |\O|.
\end{equation*}
Finally by \eqref{e3b}, \eqref{e3alt} and \eqref{a5} we have that
\begin{equation} \label{a15}
\cT(\O) = \int_0^\infty H_{\O}(t)\,dt,
\end{equation}
i.e., \emph{the torsional rigidity is the integral of the heat content}.

\subsection{Torsional rigidity of the complement of Brownian motion}
\label{torsrig}

In the present paper we consider the flat unit torus $\T^m$. Let $(\beta(s),s \geq 0; \P_x,
x\in \T^m)$ be a second independent Brownian motion on $\T^m$. Our object of interest
is the random set (see Fig.~\ref{BM2})
\begin{equation*}
%\label{e3insert}
\cB(t) = \T^m \backslash \beta[0,t].
\end{equation*}
In particular, we are interested in the \emph{expected torsional rigidity of} $\cB(t)$:
\begin{equation}
\label{e4}
\spadesuit(t) = \E_0\big(\cT\big(\cB(t)\big)\big), \qquad t \geq 0.
\end{equation}
Since $|\T^m|=1$ and $|\beta[0,t]|=0$, the torsional rigidity is the expected time needed by
the first Brownian motion $\tilde\beta$ to hit $\beta[0,t]$ averaged over all starting points in
$\T^m$. As $t\to\infty$, $\beta[0,t]$ tends to fill $\T^m$. Hence we expect that $\lim_{t\to\infty}
\spadesuit(t) = 0$. The results in this paper identify the speed of convergence. This speed
provides information on the random geometry of $\cB(t)$. In earlier work \cite{vdBBdH} we
considered the inradius of $\cB(t)$.

%%%%%%%%%%%%%%%%%%%%%%%%%%%
\begin{figure}[htbp]
\vspace{0.5cm}
\begin{center}
\includegraphics[width=0.35\linewidth]{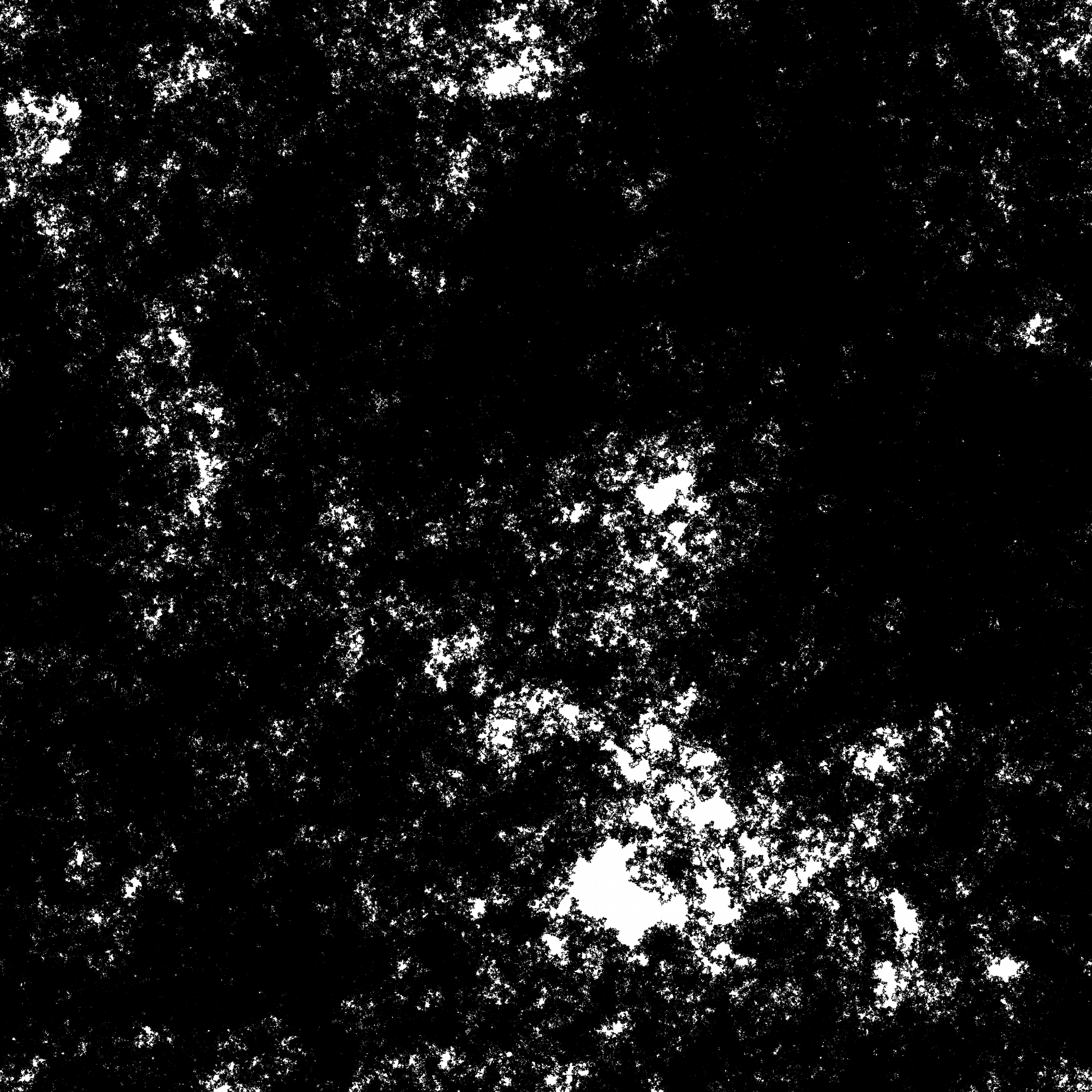}
\end{center}
\label{BM2}
\caption{Simulation of $\beta[0,t]$ for $t=15$ and $m=2$. The Brownian path $\beta[0,t]$
is black, its complement $\cB(t)=\T^m\backslash\beta[0,t]$ is white.}
\end{figure}
%%%%%%%%%%%%%%%%%%%%%%%%%%%

The case $m=1$ is uninteresting. For $m=2$, as $t$ gets large the set $\cB(t)$
decomposes into a large number of disjoint small components (see Fig.~\ref{BM2}),
while for $m \geq 3$ it remains connected. As shown in \cite{GdH}, in the latter
case $\cB(t)$ consists of ``lakes'' connected by ``narrow channels'', so that we may
think of it as a \emph{porous medium}. Below we identify the asymptotic behaviour
of $\spadesuit(t)$ as $t\to\infty$ when $m=2,3$.

For $m \geq 4$ we have $\spadesuit(t)=\infty$ for all $t \geq 0$ because Brownian
motion is polar. To get a non-trivial scaling, the Brownian path must be thickened to
a \emph{shrinking Wiener sausage}
\begin{equation}
\label{e12}
W_{r(t)}[0,t]
= \big\{x\in\T^m\colon\,d_t(x) \leq r(t)\big\},  \qquad t>0,
\end{equation}
where $r\colon\,(0,\infty) \to (0,\infty)$ is such that $\lim_{t\to\infty} t^{1/(m-2)} r(t)=0$.
This choice of shrinking is appropriate because for $m \geq 3$ typical regions in $\cB(t)$
have a size of order $t^{-1/(m-2)}$ (see \cite{DPR} and \cite{GdH}). The object of interest
is the random set
\begin{equation*}
\cB_{r(t)}(t) = \T^m \backslash W_{r(t)}[0,t],
\end{equation*}
in particular, the \emph{expected torsional rigidity of} $\cB_{r(t)}(t)$:
\begin{equation*}
\spadesuit_{r(t)}(t) = \E_0\big(\cT\big(\cB_{r(t)}(t)\big)\big), \qquad t > 0.
\end{equation*}
Below we identify the asymptotic behaviour of $\spadesuit_{r(t)}(t)$ as $t\to\infty$ for
$m \geq 4$ subject to a condition under which $r(t)$ does not decay too fast.

%%%%

\subsection{Asymptotic scaling of expected torsional rigidity}
\label{asympscal}

Theorems~\ref{the1}--\ref{the3} below are our main results for the scaling of $\spadesuit(t)$
and $\spadesuit_{r(t)}(t)$ as $t\to\infty$. In what follows we write $f \asymp g$ when
$0 < c \leq f(t)/g(t) \leq C < \infty$ for $t$ large enough.

\begin{theorem}
\label{the1}
If $m=2$, then
\begin{equation}
\label{e5}
\spadesuit(t) \asymp t^{1/4}\,e^{-4(\pi t)^{1/2}}, \qquad  t\to\infty.
\end{equation}
\end{theorem}

\begin{theorem}
\label{the2}
If $m=3$, then
\begin{equation}
\label{e6}
\spadesuit(t) = [1+o(1)]\, \frac{2}{t^2}\,\E_0\left(\frac{1}{\cp(\beta[0,1])^2}\right),
\qquad t\to\infty,
\end{equation}
where $\cp(\beta[0,1])$ is the Newtonian capacity of $\beta[0,1]$ in $\R^3$.
All inverse moments of $\cp(\beta[0,1])$ are finite.
\end{theorem}

\begin{theorem}
\label{the3}
If $m \geq 4$ and
\begin{equation}
\label{epscond}
\lim_{t\to\infty} t^{1/(m-2)} r(t)=0, \qquad
\left\{\begin{array}{ll}
m=4\colon &\lim\limits_{t\to\infty}  \frac{t}{\log^3 t}\,\frac{1}{\log(1/r(t))} = \infty,\\[0.2cm]
m\geq 5\colon  &\lim\limits_{t\to\infty} \frac{t}{\log^3 t}\,r(t)^{m-4} = \infty,
\end{array}
\right.
\end{equation}
then
\begin{equation}
\label{e13}
\spadesuit_{r(t)}(t) = [1+o(1)]\, \frac{1}{\kappa_m\,t^{2/(m-2)}}\,
\E_0\left(\frac{1}{\cp(W_{\eps(t)}[0,1])}\right), \qquad t \to \infty,
\end{equation}
where $\eps(t) = t^{1/(m-2)} r(t)$, $\cp(W_\eps[0,1])$ is the Newtonian capacity of
$W_\eps[0,1]$ in $\R^m$, and where $\kappa_m$ is the Newtonian capacity of the ball with radius $1$ in $\R^m$,
\begin{equation}
\label{e10}
\kappa_m = 4\pi^{m/2} \Big/\,\Gamma\left(\frac{m-2}{2}\right).
\end{equation}
All inverse moments of $\cp(W_\eps[0,1])$ are finite for all
$\eps>0$.
\end{theorem}

We expect that similar results hold when $\T^m$ is replaced by a smooth $m$-dimensional
compact connected Riemannian manifold without boundary. We further expect that the
torsional rigidity satisfies a strong law of large numbers for $m \geq 3$ but not for $m=2$.

A key ingredient in the proof of Theorem~\ref{the3} is the following scaling behaviour of
the capacity of the Wiener sausage for $m \geq 4$. Let
\begin{equation}
\label{defscalcap}
\cC(t) = \left\{\begin{array}{ll}
{\displaystyle\frac{\log t}{t}}\,\cp(W_1[0,t]), &m = 4,\\[0.2cm]
{\displaystyle\frac{1}{t}}\,\cp(W_1[0,t]), &m \geq 5.
\end{array}
\right.
\end{equation}
Then there exist constants $c_m \in (0,\infty)$, $m \geq 4$, such that
\begin{equation}
\label{capscalm4}
\cC(t) = [1+o(1)]\,c_m \quad \text{ in $\P_0$-probability as } t \to \infty.
\end{equation}
In Section~\ref{scalCAP} we prove \eqref{capscalm4} for $m \geq 5$ with the help of
subadditivity. For $m=4$, \eqref{capscalm4} is proven in \cite{ASS}.

%%%%

\subsection{Discussion}
\label{disc}

We refer the reader to \cite{GdH} and \cite{BK} for an overview of what is known about
the geometry of the complement of Brownian motion on the unit torus.

\medskip\noindent
{\bf 1.}
Theorems~\ref{the1} and \ref{the2} identify the scaling of the expected torsional rigidity
in low dimensions.
 This scaling may be viewed in the following context. Let $d(x,y)$
denote the distance between $x,y \in \T^m$. The distance of $x$ to $\beta[0,t]$ is denoted by
\begin{equation}
\label{e7}
d_t(x) = \min_{y\in \beta[0,t]} d(x,y).
\end{equation}
The \emph{inradius} of $\cB(t)$ is the random variable $\rho_t$ defined by
\begin{equation*}
%\label{e8}
\rho_t = \max_{x\in\T^m} d_t(x).
\end{equation*}
A detailed analysis of $\rho_t$ and related quantities was given in \cite{DPRZ}, \cite{BK} for
$m=2$ and in \cite{DPR}, \cite{GdH} for $m \geq 3$. In \cite{vdBBdH} it was shown that for
$m=2$,
\begin{equation}
\label{e11}
\E_0(\rho_t) = e^{- (\pi t)^{1/2}[1+o(1)]}, \qquad t\to\infty,
\end{equation}
while for $m \geq 3$,
\begin{equation}
\label{e9}
\E_0(\rho_t) = [1+o(1)]\,\left(\frac{m}{(m-2)\kappa_m}\,\frac{\log t}{t}\right)^{1/(m-2)}, \qquad t\to\infty.
\end{equation}

A ball of radius $r$ in $\T^m$ with $r$ sufficiently small has a torsional rigidity proportional
to $r^{m+2}$. Theorem~\ref{the1} and \eqref{e11} show that $\log\spadesuit(t) = - [1+o(1)]\,
4(\pi t)^{1/2} = [1+o(1)]\, \log \E_0(\rho_t)^4$ for $m=2$, while Theorem~\ref{the2} and
\eqref{e9} show that $\spadesuit(t) \asymp t^{-2} \gg \E_0(\rho_t)^5$ for $m=3$. Thus, for
$m=2$ the main contribution to the asymptotic behaviour of $\log\spadesuit(t)$ comes from
the components in $\cB(t)$ that have a size of order $\rho_t$ (which are atypical; see \cite{DPRZ}
and \cite{BK}), while for $m=3$ the main contribution to the asymptotic behaviour of $\spadesuit(t)$
comes from regions in $\cB(t)$ that have a size of order $t^{-1}$ (which are typical; see \cite{DPR}
and \cite{GdH}), i.e., most of $\cB(t)$ contributes.

\medskip\noindent
{\bf 2.}
For $m=2$ it is shown in \cite{BK} that
\begin{equation}
\label{BKscal}
\rho_t = t^{-1/8+o(1)}\,e^{-(\pi t)^{1/2}} \quad \text{in $\P_0$-probability}, \qquad t \to \infty,
\end{equation}
which is a considerable sharpening of \eqref{e11}. The proof is long and difficult.
Combining \eqref{BKscal} with what we found in Theorem~\ref{the1}, we get the
relation
\begin{equation}
\label{scalconj}
\spadesuit(t) \asymp t^{3/4+o(1)}\,\E_0(\rho_t)^4,
\end{equation}
provided \eqref{BKscal} also holds in mean (which is expected but has not been proved).
Clearly, $\spadesuit(t)$ is not dominated by the largest component in $\cB(t)$ alone:
smaller components contribute too as long as they have a comparable size. The scaling in
\eqref{scalconj} suggests that the number of such components is of order $t^{3/4+o(1)}$.
In order to settle this issue, we would need to strengthen Theorem~\ref{the1} to tightness.

\medskip\noindent
{\bf 3.}
Theorem~\ref{the3} identifies the scaling of the expected torsional rigidity in high dimensions.
Via the scaling relation in distribution
\begin{equation}
\label{capscalapr}
\cp(W_\eps[0,1]) = \cp(\eps W_1[0, \eps^{-2}]) = \eps^{m-2}\cp(W_1[0, \eps^{-2}]),
\qquad \eps>0,
\end{equation}
it follows from \eqref{defscalcap}--\eqref{capscalm4} that $\cp(W_\eps[0,1])=[1+o(1)]\,c_m\eps^{m-4}$
in $\P_0$-probability as $\eps \downarrow 0$ when $m \geq 5$. In that case Theorem~\ref{the3}
yields the asymptotics
\begin{equation}
\label{scalreg}
\spadesuit_{r(t)}(t) = [1+o(1)]\,\frac{1}{\kappa_m c_m\,t\, r(t)^{m-4}} , \qquad t \to \infty.
\end{equation}
It also follows from \eqref{defscalcap}--\eqref{capscalm4} that $\cp(W_\eps[0,1])=[1+o(1)]\,
c_4/2\log(1/\eps)$ in $\P_0$-probability as $\eps \downarrow 0$ when $m = 4$. In that
case Theorem~\ref{the3} yields the asymptotics
\begin{equation}
\label{scalreg*}
\spadesuit_{r(t)}(t) = [1+o(1)]\,\frac{2\log(1/t^{1/2}r(t))}{\kappa_4 c_4\,t}, \qquad t \to \infty.
\end{equation}
By the second half of \eqref{epscond}, both \eqref{scalreg} and \eqref{scalreg*} correspond to
the regime where $\spadesuit_{r(t)}(t) = o(1/\log^3 t)$. We have not attempted to improve this
to $o(1)$.

\medskip\noindent
{\bf 4.}
We did not investigate the regime for $m \geq 4$ where $r(t)$ decays so fast that $\spadesuit_{r(t)}(t)$
diverges as $t\to\infty$. In that regime, the Brownian motion $\tilde\beta$ in \eqref{e1} runs
around $\T^m$ many times before it hits $W_{r(t)}[0,t]$, and the growth of $\spadesuit_{r(t)}(t)$ depends
on the global rather than the local properties of $W_{r(t)}[0,t]$.

\medskip\noindent
{\bf 5.}
We saw in Section~\ref{background} that the torsional rigidity is closely related to the principal Dirichlet
eigenvalue. In Section~\ref{torsion} we will exhibit a relation with the square-integrated distance
function and the largest inradius. In Section~\ref{extra} we will give a quick proof of the following
inequality relating the torsional rigidity to
\begin{equation}
\label{prinDir}
\lambda_1\big(\cB(t)\big), \quad \lambda_1\big(\cB_{r(t)}(t)\big),
\end{equation}
the principal Dirichlet eigenvalue of $\cB(t)$ for $m=2,3$ and $\cB_{r(t)}(t)$ for $m \geq 4$.

\begin{theorem}
\label{the4}
(a) If $m=2,3$, then for $t$ large enough,
\begin{equation*}
\E_0\big(\lambda_1\big(\cB(t)\big)\big) \geq \spadesuit(t)^{-2/(m+2)}.
\end{equation*}
(b) If $m \geq 4$ and $\lim_{t \to \infty} \spadesuit_{r(t)}(t) = 0$, then for $t$ large enough,
\begin{equation*}
\E_0\big(\lambda_1\big(\cB_{r(t)}(t)\big)\big) \geq \spadesuit_{r(t)}(t)^{-2/(m+2)}.
\end{equation*}
\end{theorem}

\noindent
Combining the result for $m=2$ with what we found in Theorem~\ref{the1}, we obtain
\begin{equation}
\label{eigvalbd*}
\E_0\big(\lambda_1\big(\cB(t)\big)\big) \succeq t^{-1/8}\,e^{2(\pi t)^{1/2}},
\end{equation}
where $f \succeq g$ means that $f(t)/g(t) \geq c > 0$ for $t$ large enough. In \cite{vdBBdH}
we conjectured that $\log \E_0(\lambda_1(\cB(t)))$ $= [1+o(1)]\, 2(\pi t)^{1/2}$, which fits the
lower bound in \eqref{eigvalbd*}. However, a better estimate than \eqref{eigvalbd*} is possible.
Namely, in Section~\ref{torsion} we will see that $\lambda_1(\cB(t)) \asymp 1/\rho_t^2$, and
so Jensen's inequality gives the lower bound $\E_0(\lambda_1(\cB(t)) \geq 1/\E_0(\rho_t)^2$.
Assuming that the scaling in \eqref{BKscal} also holds in mean (which is expected but has not
been proved), we get
\begin{equation}
\label{BKlambda1}
\E_0\big(\lambda_1\big(\cB(t)\big)\big)
\succeq t^{1/4+o(1)}\,e^{2(\pi t)^{1/2}},
\end{equation}
which is better than \eqref{eigvalbd*} by a factor $t^{3/8+o(1)}$. Presumably \eqref{BKlambda1}
captures the correct scaling behaviour.

%%%

\subsection{Brief sketch and outline}
\label{sketch}

For $m=2$, $\cB(t)$ consists of countably many connected component and the expected
lifetime is sensitive to the starting point. We make use of the Hardy inequality to relate the
time-integrated heat content to the space integral $\int_{\T^2} \mathrm{dist}(x,\beta[0,t])^2\,dx$.
Because of the symmetry of $\T^2$, the problem boils down to studying the distribution of
$\mathrm{dist}(x,\beta[0,t])^2$ with $x\in\T^2$ chosen uniformly at random. This can be
done by using a domain perturbation formula for the Dirichlet Laplacian eigenvalues.

For $m \geq 3$, $\cB(t)$ has only one connected component and the proof is probabilistic.
The starting point is the representation
\begin{equation*}
\spadesuit(t) = \int_0^\infty ds\,\,(\P\otimes\tilde\P)\big(\beta [0, t] \cap \tilde\beta[0,s]
= \emptyset\big).
\end{equation*}
It is easy to see that $\tilde\beta$ hits $\beta[0,t]$ within time $o((\log t)^{-1}$) with a very
high probability. For $s \leq (\log t)^{-1}$, the above integrand is the probability that $\beta$
avoids the small set $\tilde\beta[0,s]$ for a long time $t$. We appeal to a recursive
argument to evaluate this probability. Roughly speaking, in each unit of time $\beta$ hits
$\tilde\beta[0,s]$ with probability $\approx \cp(\tilde\beta[0,s])$.

\medskip\noindent
{\bf Outline.}
The remainder of this paper is organised as follows. In Section~\ref{torsion} we recall some
analytical facts about the torsional rigidity. In Sections~\ref{proofthe1}--\ref{proofthe3} we
prove Theorems~\ref{the1}--\ref{the3}, respectively. The proof of Theorem~\ref{the4} is given
in Section~\ref{extra}, while the proof of the scaling in \eqref{defscalcap}--\eqref{capscalm4}
for $m \geq 5$ is given in Section~\ref{scalCAP}.

%%%%%%%%%%%%%%%%%%%%%%%%%%%%%%%%%%%%%%%%

\section{Analytical facts for the torsional rigidity}
\label{torsion}

Let $M$ be an $m$-dimensional Riemannian manifold without boundary that is both
geodesically and stochastically complete. In most of this paper we focus on the case
where $M$ is the $m$-dimensional unit torus $\T^m$. However, the results mentioned
below hold in greater generality.  We derive certain a priori estimates on the torsional
rigidity that will be needed later on.

For an open set $\O
\subset M$ with boundary $\partial \O$, and with finite Lebesgue measure $|\O|$, we denote the Dirichlet heat kernel by $p_{\O}(x,y;t)$,
$x,y\in\O$, $t>0$. Recall that the Dirichlet heat kernel is non-negative, monotone in $\Omega$, symmetric. Thus, we have that
\begin{equation*}
%\label{a6}
0 \leq p_\O(x,y;t) \leq p_{M}(x,y;t).
\end{equation*}

Since $|\O|<\infty$, there exists an $L^2(\O)$ eigenfunction expansion for the Dirichlet heat kernel in terms
of the Dirichlet eigenvalues $\lambda_1(\O) \leq \lambda_2(\O) \leq \cdots$, and a corresponding
orthonormal set of eigenfunctions $\varphi_1,\varphi_2,\cdots$ in $L^2(\O)$:
\begin{equation}
\label{a9}
p_\O(x,y;t) = \sum_{j\in\N} e^{-t\lambda_j(\O)} \varphi_j(x)\varphi_j(y).
\end{equation}
Since
\begin{equation*}%\label{a1}
u_\O(x;t) = \int_\O p_\O(x,y;t)\,dy,
\end{equation*}
we have that
\begin{equation*}
%\label{c1alt}
v_{\O}(x)=\int_{\O} dy \int_0^{\infty} dt\,\,p_{\O}(x,y;t),
\end{equation*}
and
\begin{equation}\label{x1}
\cT(\O)=\int_0^{\infty}dt\,\int_{\O}dx\,\int_{\O}dy\,p_{\O}(x,y;t).
\end{equation}

 Lemma~\ref{lemlb*} below provides
an upper bound on the Dirichlet eigenfunctions in terms of the Dirichlet eigenvalues. This
bound will show that the eigenfunctions are in $L^{\infty}(\T^m)$, which by H\"older's inequality
implies that they are in $L^p(\T^m)$ for all $1 \leq p \leq \infty$. Lemma~\ref{lemlb} below states
upper and lower bounds on the torsional rigidity that will be needed later on.

\begin{lemma}
\label{lemlb*}
Suppose that $\O\subset M$, $|\O|<\infty$, $\sup_{x\in M} p(x,x;t)<\infty$ for all $t>0$. Then
\begin{equation}
\label{a161}
\left\Vert\varphi_j\right\Vert_{L^{\infty}(\O)}^2 \leq
e \sup_{x \in M} p_M(x,x;\lambda_j(\O)^{-1}), \qquad  j \in \N.
\end{equation}
\end{lemma}

\begin{proof}
By \eqref{a9} and the domain monotonicity of the Dirichlet heat kernel (\cite{AG}), we have that
\begin{equation}
\label{a162}
\varphi_j(x)^2 \leq e\,p_\O(x,x;\lambda_j(\O)^{-1}) \leq e\,p_M(x,x;\lambda_j(\O)^{-1}).
\end{equation}
Taking first the supremum over $x\in M$ in the right-hand side of \eqref{a162} and subsequently
in the left-hand side of \eqref{a162}, we get \eqref{a161}.
\end{proof}

 Let
\begin{equation}
\label{dtxdef}
\delta_\O(x) =\min_{y\in \R^m \backslash\O} d(x,y)
\end{equation}
denote the distance of $x \in \O$ to $\R^m\backslash\O$.

\begin{lemma}
\label{lemlb}
(a) Let $M$ be a Riemannian manifold that is both geodesically and stochastically complete.
Let $\O$ be an open subset of $M$ with $|\O|<\infty$. Then
\begin{equation}
\label{l1}
\cT(\Omega) \leq \lambda_1(\Omega)^{-1}|\Omega|.
\end{equation}
(b) Suppose that $M$ and $\O$ satisfy the hypotheses in (a). Then
\begin{equation}
\label{la}
\cT(\O)\ge \lambda_1(\O)^{-1}\lVert \varphi_1
\rVert_{L^{\infty}(\O)}^{-2}.
\end{equation}
(c) Let $\O\subset \R^m$. Then
\begin{equation}
\label{l2}
\cT(\O) \geq \frac{1}{2m} \int_\O \delta_\O(x)^2\,dx.
\end{equation}
(d) Let $\O\subset \R^2$ be simply connected and $\delta_\O\in
L^2(\O)$. Then
\begin{equation}
\label{u1}
\cT(\O) \leq 16 \int_\O \delta_\O(x)^2\,dx.
\end{equation}
(e) Let $\O\subset\T^m$. Then $\O$ can be embedded in $\R^m$ if and
only if $\max_{i=1}^m |x_i-y_i| \leq \frac{1}{2}$ for all
$x=(x_1,\ldots,x_m) \in \O$ and $y=(y_1,\ldots,y_m) \in \O$. If
$\O\subset \T^2$ can be embedded in $\R^2$, then
\begin{equation}
\label{sc}
\frac14 \int_\O \delta_\O(x)^2\,dx \leq \cT(\O) \leq 16\int_\O \delta_\O(x)^2\,dx.
\end{equation}
\end{lemma}

\begin{proof}
(a) Since the eigenfunctions are in all $L^p(\O)$, we have by \eqref{a9}, \eqref{x1} and
Parseval's identity that
\begin{equation}
\label{a10}
\cT(\O) = \int_0^{\infty}dt\,\sum_{j\in\N} e^{-t\lambda_j(\O)} \left(\int_\O \varphi_j\right)^2
\leq \lambda_1(\O)^{-1} \sum_{j\in\N}\left(\int_\O \varphi_j\right)^2=\lambda_1(\O)^{-1}|\O|.
\end{equation}

Inequality \eqref{l1} goes back to \cite{PSZ}. For a recent discussion and further improvements we refer the reader
to \cite{vdBNTV}.\\
(b) By \eqref{a15} andthe first identity in \eqref{a10}, we have that
\begin{equation}
\label{a1900}
\cT(\O) \geq \int_0^\infty e^{-t\lambda_1(\O)}\,dt \left(\int_\O \varphi_1\right)^2
= \lambda_1(\O)^{-1} \left(\int_\O \varphi_1\right)^2.
\end{equation}
By Lemma \ref{lemlb*}, we have that $\lVert \varphi_1 \rVert_{L^{\infty}(\O)}<\infty$, and so
\begin{equation}\label{a1899}
1= \int_\O \varphi_1^2 \leq \lVert \varphi_1\rVert_{L^{\infty}(\O)}\int_\O |\varphi_1|.
\end{equation}
Inequality \eqref{la} follows from \eqref{a1900},\eqref{a1899}, and the fact that $\varphi_1$
does not change sign.\\
(c) For every $x \in \O$ the open ball $B_{\delta_{\O}(x)}(x)$ with centre $x$ and radius
$\delta_\Omega(x)$ is contained in $\O$. Therefore, by domain monotonicity, the expected
life time satisfies $v_{\O}(y) \geq v_{B_{\delta(x)}(x)}(y)$. Hence
\begin{equation*}
%\label{a1901}
v_{\O}(y) \geq v_{B_{\delta_{\O}(x)}(x)}(y) = \frac{\delta_{\O}(x)^2-|x-y|^2}{2m},
\qquad |y-x| \leq \delta_{\O}(x).
\end{equation*}
Choose  $y=x$, integrate over $x \in \O$ and use \eqref{e3alt}, to get the claim.\\
(d) It was shown in \cite{A} that the Dirichlet Laplacian on a simply connected proper subset of $\R^2$
satisfies a strong Hardy inequality:
\begin{equation*}
%\label{a191}
\int_\O \vert \nabla w(x) \vert^2\,dx \geq \frac{1}{16} \int_\O \frac{w(x)^2}{\delta_\O(x)^2}\,dx
\qquad \forall\, w \in C_c^\infty(\O).
\end{equation*}
Theorem 1.5 in \cite{vdBG} implies \eqref{u1}.\\
(e) Recall that the metric on $\T^m$ is given by
\begin{equation*}
d(x,y)=\left(\sum_{i=1}^m\min\big\{|x_i-y_i|,1-|x_i-y_i|\big\}^2\right)^{1/2}.
\end{equation*}
Note that $\textup{diam}(\T^m)=\tfrac12\sqrt{m}$ because $\min\{|x_i-y_i|,1-|x_i-y_i|\} \leq \tfrac12$.
If $|x_i-y_i| \leq \tfrac12$ for all $i$, then $d(x,y)=|x-y|$. Next, suppose that $d(x,y)=|x-y|$. Then
$\sum_{i=1}^m\min\{|x_i-y_i|,1-|x_i-y_i|\}^2=\sum_{i=1}^m|x_i-y_i|^2$. Let $I = \{i\colon\,|x_i-y_i|>\tfrac12\}$.
Then $\sum_{i\in I} (1-2|x_i-y_i|) = 0$. We therefore conclude that $I=\emptyset$. Finally, \eqref{sc}
follows from \eqref{l2} for $m=2$ and \eqref{u1}.
\end{proof}

%%%%%%%%%%%%%%%%%%%%%%%%%%%%%%%%%

\section{Torsional rigidity for $m=2$}
\label{proofthe1}

In Section~\ref{invDEV} we show that the inverse of the principal Dirichlet eigenvalue
of $\cB(1)=\T^2\backslash\beta[0,1]$ has a finite exponential moment. In Section~\ref{secproofthe1}
we use this result to prove Theorem~\ref{the1}.

%%%%

\subsection{Exponential moment of the inverse principal Dirichlet eigenvalue}
\label{invDEV}

\begin{lemma}
\label{lambdaest}
There exists $c > 0$ such that
\begin{equation*}
\E_0\left(\exp\left[\frac{c}{\lambda_1(\cB(1))}\right]\right) < \infty.
\end{equation*}
\end{lemma}

\begin{proof}
Let $\cp(A)$ denote the logarithmic capacity of a measurable set $A \subset \R^2$. It is
well known (see \cite{MNP}) that if $\cp(A)>0$ and $\epsilon A$ is a homothety of $A$ by
a factor $\epsilon$, then
\begin{equation*}
%\label{c1}
\cp(\epsilon A) = \frac{2\pi}{\log(1/\eps)}\,[1+o(1)],
\qquad \eps\downarrow 0,
\end{equation*}
and
\begin{equation*}
%\label{c2}
\lambda_1(\T^2\backslash\epsilon A) = \frac{2\pi}{\log(1/\eps)}\,[1+o(1)],
\qquad \eps\downarrow 0.
\end{equation*}
In particular, if $L_\eps$ is a straight line segment of length $\eps$, then there exists
a $c' \in (0,\infty)$ such that
\begin{equation*}
%\label{c3}
\lambda_1(\T^2 \backslash L_\eps) \geq \frac{c'}{\log(1/\eps)},
\qquad 0< \eps \leq \tfrac12.
\end{equation*}
Since $\cp(\beta[0,1]) \geq \cp(L_{|\beta(1)|}) \geq \cp(L_{(\tfrac12 \wedge |\beta(1)|)})$,
we get
\begin{align*}
\E_0&\left(\exp\left[\frac{c}{\lambda_1(\cB(1))}\right]\right)
\leq \E_0\left( (\tfrac12 \wedge |\beta(1)|)^{-c/c'} \right)
\leq (\tfrac12)^{-c/c'} + \E_0\left( |\beta(1)|^{-c/c'} \right)\nonumber \\ &
=  (\tfrac12)^{-c/c'} + \int_{\R^2} |x|^{-c/c'}\,\frac{1}{4\pi}\,e^{-|x|^2/4}\,dx,
\end{align*}
which is finite when $c/c'<2$.
\end{proof}

%%%%

\subsection{Proof of Theorem~\ref{the1}}
\label{secproofthe1}

\begin{proof}
The proof comes in 6 Steps, and is based on Lemmas~\ref{pklem1}--\ref{bbl} below.
We use the following abbreviations (recall \eqref{e7} and \eqref{prinDir}):
\begin{equation}
\label{abbrdef}
D^2_t = \int_{\T^2} d_t(x)^2\,dx,
\qquad
\lambda_t = \lambda_1(\cB(t)).
\end{equation}

\medskip\noindent
{\bf 1.}
Note that $\beta[0,t]$ is a closed subset of $\T^2$ a.s. Hence $\cB(t)$ is open and its
components are open and countable. Let $\{\Omega_1(t),\Omega_2(t),\cdots\}$ enumerate
these components. Let
\begin{equation*}
%\label{e5.57}
\phi_i(t) = \mathrm{diam}(\Omega_i(t)) = \sup_{x,y\in \Omega_i(t)} d(x,y),
\end{equation*}
and abbreviate
\begin{equation*}
\cI_u(t) =  \{i\in\N\colon\,\phi_i(t) \leq u\}, \quad
\cE_u(t) = \left\{\sup_{i\in\N} \phi_i(t)>u\right\}, \qquad u \in (0,1).
\end{equation*}
It follows from the proof of Lemma~\ref{lemlb}(d) that if $i \in \cI_{1/2}(t)$, then $\Omega_i(t)$
can be isometrically embedded in $\R^2$. Since $\beta[0,t]$ is continuous a.s., each $\Omega_i(t)$
is simply connected. Since the torsional rigidity is additive on disjoint sets we have that
\begin{align}
\label{e5.58}
\cT(\cB(t)) = \sum_{i\in \N} \cT(\O_i(t))
= \sum_{i \in \cI_{1/2}(t)} \cT(\O_i(t))
+ \sum_{i \notin \cI_{1/2}(t)} \cT(\O_i(t)).
\end{align}

\medskip\noindent
{\bf 2.} The first term in the right-hand side of \eqref{e5.58} is estimated from above by
Lemma~\ref{lemlb}(d). This gives (recall \eqref{dtxdef})
\begin{equation*}
%\label{e5.59}
\sum_{i \in \cI_{1/2}(t)} \cT(\O_i(t))
\leq 16 \sum_{i \in \cI_{1/2}(t)} \int_{\O_i(t)} \delta_{\O_i(t)}(x)^2\,dx
\leq 16 \sum_{i\in \N} \int_{\O_i(t)} \delta_{\O_i(t)}(x)^2\,dx = 16 D^2_t.
\end{equation*}
The second term in the right-hand side of \eqref{e5.58} is estimated from above by
Lemma~\ref{lemlb}(a). This gives
\begin{equation*}
%\label{e5.60}
\sum_{i \notin \cI_{1/2}(t)} \cT(\O_i(t)) \leq \sum_{i \notin \cI_{1/2}(t)}
\lambda_t^{-1}\,|\O_i(t)|
\leq 1_{\cE_{1/2}(t)}\, \lambda_t^{-1} \sum_{i\in\N}  |\O_i(t)|
= 1_{\cE_{1/2}(t)}\, \lambda_t^{-1}.
\end{equation*}
By Cauchy-Schwarz, this term contributes to $\spadesuit(t)$ at most
\begin{align}
\label{e5.61}
\E_0\left(1_{\cE_{1/2}(t)}\,\lambda_t^{-1}\right)
\leq \left(\P_0(\cE_{1/2}(t))\right)^{1/2}
\Big(\E_0\big(\lambda_t^{-2}\big)\Big)^{1/2}.
\end{align}
To bound the probability in the right-hand side of (\ref{e5.61}) from above, we let
$\{Q_1,\dots, Q_N\}$, $N=10^4$, be any open disjoint collection of squares in
$\T^2,$ each with area $10^{-4}$ and not containing $0$. Furthermore, we let
$\bar Q_{N,\epsilon}$ be the open $\epsilon$-neighbourhood of the union of the
boundaries of these squares with $\epsilon=10^{-3}$. Then $\beta[0,1]$ starting
at $0$ has a positive probability $p'=p'(N,\epsilon)$ of making a closed loop around
each of these squares and staying inside $\bar Q_{N,\epsilon}$. Translating $\{Q_1,\dots,
Q_N\}$ such that these squares do not contain $\beta(1)$, we find that $\beta[1,2]$
starting at $\beta(1)$ has a positive probability $p'$ of making a closed loop around
each of these translated squares and staying inside $\bar Q_{N,\epsilon}+\beta(1)$.
Continuing this way, by induction we find that the probability of $\beta[0,t]$ not
making any of these closed translated loops is at most $(1-p')^{\lfloor t\rfloor}$,
where $\lfloor\cdot\rfloor$ denotes the integer part. Hence $\P_0(\sup_{i\in\N}
\phi_i(t)>\tfrac12)\leq(1-p')^{\lfloor t\rfloor}$, and so
\begin{equation}
\label{e5.62}
\P_0(\cE_{1/2}(t)) \leq e^{-pt}, \qquad t\geq 2,
\end{equation}
for some $p>0$. We conclude that
\begin{equation}
\label{e5.63}
\spadesuit(t) \leq 16\,\E_0\left(D^2_t\right)
+ e^{-pt/2}\Big(\E_0(\lambda_t^{-2})\Big)^{1/2},
\qquad t \geq 2.
\end{equation}
Since $t \mapsto \lambda_t$ is non-decreasing, Lemma~\ref{lambdaest}
implies that the second term decays exponentially fast in $t,$ and therefore is harmless for the upper
bound in \eqref{e5}.

\medskip\noindent
{\bf 3.}
To derive a lower bound for $\spadesuit(t)$, we note that by Lemma~\ref{lemlb}(e) we have
\begin{align*}
%\label{e5.64}
&\cT(\cB(t))
= \sum_{i\in\N} \cT(\O_i(t))
\geq \sum_{i \in \cI_{1/2}(t)} \cT(\O_i(t))\nonumber\\
&\geq \tfrac14 \sum_{i \in \cI_{1/2}(t)} \int_{\O_i(t)}
\delta_{\O_i(t)}(x)^2\,dx
\geq \tfrac14 \sum_{i\in\N} \int_{\O_i(t)} \delta_{\O_i(t)}(x)^2\,dx
-\tfrac14 \sum_{i \notin \cI_{1/2}(t)} \int_{\O_i(t)} \delta_{\O_i(t)}(x)^2\,dx\nonumber\\
&\geq \tfrac14 D^2_t
-\tfrac14 \sum_{i \notin \cI_{1/2}(t)} 1_{\cE_{1/2}(t)}
\int_{\O_i(t)} \delta_{\O_i(t)}(x)^2\,dx
\geq \tfrac14 D^2_t
-\tfrac18 1_{\cE_{1/2}(t)},
\end{align*}
where in the last inequality we use that $ \delta_{\O_i(t)}(x) \leq \textup{diam}(\T^2)=\tfrac12\sqrt{2}$
and $|\T^2|=1$. We conclude by \eqref{e5.62} that
\begin{equation}
\label{e5.65}
\spadesuit(t) \geq \tfrac14 \E_0\left(D^2_t\right)
- e^{-pt}, \qquad t\geq 2.
\end{equation}
The second term is again harmless for the lower bound in \eqref{e5}.

\medskip\noindent
{\bf 4.}
The estimates in \eqref{e5.63} and \eqref{e5.65} show that $\spadesuit(t) \asymp \E_0(D^2_t)$
up to exponentially small error terms. In order to obtain the leading order asymptotic behaviour
of $\E_0(D^2_t)$, we make a \emph{dyadic partition} of $\T^2$ into squares as follows. Partition
$\T^2$ into four $1$-squares of area $\tfrac14$ each. Proceed by induction to partition each
$k$-square into four $(k+1)$-squares, etc. In this way, for each $k \in \N$, $\T^2$ is partitioned into
$2^{2k}$ $k$-squares. We define a $k$-square to be \emph{good} when the path $\beta[0,t]$ does
not hit this square, but does hit the unique $(k-1)$-square to which it belongs. Clearly, if $x$ belongs
to a good $k$-square, then $\textup{dist}(x,\beta[0,t]) \leq (2\sqrt{2})2^{-k}$. Hence, as the area of
each $k$-square is $2^{-2k}$, we get
\begin{equation}
\label{e5.66}
\E\left(D^2_t\right)
\leq 8 \sum_{k\in\N} 2^{-2k} \sum_{S \textup{ is a k-square}} 2^{-2k}\,\P(S \textup{ is a good square})
\leq 8 \sum_{k\in\N} 2^{-4k}\,\E\left(\# \textup{ good } k\textup{-squares}\right),
\end{equation}
where we write $\E = \int _{\T^2} dx\,\E_x$, which is the same as $\E_0$ for the quantity under
consideration, by translation invariance. To estimate the right-hand side of \eqref{e5.66} we
need three lemmas.

\begin{lemma}
\label{pklem1}
For $k\in \N$, let $p_k(t) = \P(\beta[0,t]\cap S_k)=\emptyset)$, where $S_k$ is any
of the $k$-squares. Then
\begin{equation*}
%\label{e5.67}
p_k(t) \leq e^{-t\lambda_1(\T^2 \backslash S_k)}.
\end{equation*}
\end{lemma}

\begin{proof}
Let $p_{\T^2\backslash S_k}(x,y;t)$ be the Dirichlet heat kernel for $\T^2\backslash S_k$. By the
eigenfunction expansion in \eqref{a9}, we have that
\begin{align*}
%\label{e5.68}
&p_k(t) =\int_{\T^2\backslash S_k} dx \int_{\T^2\backslash S_k} dy\,\, p_{\T^2\backslash S_k}(x,y;t)
= \int_{\T^2\backslash S_k} dx \int_{\T^2\backslash S_k} dy\,\,
\sum_{j\in\N} e^{-t\lambda_j(\T^2\backslash S_k)}
\varphi_j(x)\varphi_j(y)\nonumber\\
&\leq e^{-t\lambda_1(\T^2\backslash S_k)} \sum_{j\in\N} \left(\int_{\T^2\backslash S_k}dx\,
\varphi_j(x)\right)^2
= e^{-t\lambda_1(\T^2\backslash S_k)}|\T^2\backslash S_k|
\leq e^{-t\lambda_1(\T^2\backslash S_k)},
\end{align*}
where we use Parseval's identity in the last equality.
\end{proof}

\begin{lemma}
\label{lam1}
There exists $C<\infty$ such that, for all $k\in \N$,
\begin{equation}
\label{e5.69}
\left|\lambda_1(\T^2\backslash S_k)-\frac{2\pi}{k\log 2}\right| \leq \frac{C}{k^2}.
\end{equation}
\end{lemma}

\begin{proof}
By \cite[Theorem 1]{O} we have that, for any disc $D_{\epsilon}\subset \T^2$ with
radius $\epsilon$,
\begin{equation*}
%\label{e5.70}
\lambda_1(\T^2\backslash D_{\epsilon})
= \frac{2\pi}{\log(1/{\epsilon})} + O\big([\log(1/\epsilon)]^{-2}\big),\qquad
\epsilon \downarrow 0.
\end{equation*}
This implies, by monotonicity and continuity of $\epsilon\mapsto\lambda_1(\T^2\backslash
D_{\epsilon})$, the existence of $C'<\infty$ such that
\begin{equation}
\label{e5.71}
\left|\lambda_1(\T^2\backslash D_{\epsilon})-\frac{2\pi}{\log(1/{\epsilon})}\right|
\leq C'[\log(1/\epsilon)]^{-2},\qquad 0< \epsilon \leq \tfrac12.
\end{equation}
For $S_k\subset \T^2$ there exist two discs $D_1$ and $D_2$, with the same centre and
radii $2^{-k-1}$ and $2^{-k-1}\sqrt{2}$, such that $D_1\subset S_k\subset D_2$. Hence
$\lambda_1(\T^2 \setminus D_2) \leq \lambda_1(\T^2 \setminus S_k) \leq \lambda_1(\T^2
\setminus D_1)$, and \eqref{e5.69} follows by applying \eqref{e5.71} with $\epsilon=2^{-k-1}$
and $\epsilon=2^{-k-1}\sqrt{2}$, respectively.
\end{proof}

\begin{lemma}
\label{pklem2}
\begin{equation*}
%\label{e5.72}
\int_{\T^2} dx\,\,\P_x(S_k\ \textup{is a good $k$-square}) = p_k(t)-p_{k-1}(t).
\end{equation*}
\end{lemma}

\begin{proof}
Let $E_k$ be the event that $S_k$ is not hit. Since $S_k$ is a good $k$-square
if and only if the event $E_k \cap E_{k-1}^c$ occurs, the lemma follows because
$E_{k-1}\subset E_k$.
\end{proof}

\medskip\noindent
{\bf 5.}
We are now ready to estimate $\E(D^2_t)$. By \eqref{e5.66} and Lemma~\ref{pklem2},
\begin{equation}
\label{e5.73}
\E\left(D^2_t\right)
\leq 8 \sum_{k\in \N} 2^{-2k} \int_{\T^2} dx\,\,\P_x(S_k \textup{ is a good $k$-square})
= 8 \sum_{k\in\N} 2^{-2k}[p_k(t)-p_{k-1}(t)]
= 6 \sum_{k\in\N} 2^{-2k}p_k(t),
\end{equation}
where $p_0(t)=0$. In order to bound this sum from above we consider the contributions coming
from $k=1,\ldots K$ and $k=K+1,\ldots, \lfloor \tfrac14 t^{1/2}\rfloor$ and $k>\lfloor \tfrac14 t^{1/2}\rfloor$,
respectively, where $\lfloor\cdot\rfloor$ denotes the integer part, and we choose
\begin{equation}
\label{e5.74}
K =\lfloor (C\log 2)/\pi \rfloor
\end{equation}
with $C$ the constant in \eqref{e5.69}. Since
\begin{equation}
\label{e5.75}
\sum_{k=1}^K 2^{-2k} p_k(t) \leq \sum_{k=1}^K 2^{-2k} p_K(t)
\leq e^{-t\lambda_1(\T^2\backslash S_K)},
\end{equation}
the first contribution is exponentially small in $t$. For $k=K+1,\ldots,\lfloor \tfrac14 t^{1/2}\rfloor$
we have $C/k^2 \leq \pi/k\log 2$, and hence by Lemmas~\ref{pklem1}--\ref{lam1},
\begin{align}
\label{e5.76}
\sum_{k=K+1}^{\lfloor \tfrac14 t^{1/2}\rfloor} 2^{-2k}p_k(t)
\leq \sum_{k=K+1}^{\lfloor \tfrac14 t^{1/2}\rfloor} 2^{-2k} e^{-\frac{\pi t}{k\log 2}}
\leq \sum_{k=K+1}^{\lfloor \tfrac14 t^{1/2}\rfloor} 2^{-2k} e^{-\frac{4\pi t^{1/2}}{\log 2}}
=O(e^{-4\pi t^{1/2}}),
\end{align}
and so the second contribution is $o(t^{1/4}e^{-4(\pi t)^{1/2}})$. Finally, for $k>\lfloor \tfrac14
t^{1/2}\rfloor$ we have $e^{Ct/k^2} \le e^{16C}$, and hence
\begin{align}
\label{e5.77}
\sum_{k>\lfloor \tfrac14 t^{1/2}\rfloor} 2^{-2k}\, p_k(t)
\leq e^{16C} \sum_{k>\lfloor \tfrac14 t^{1/2}\rfloor} e^{-2k\log 2-\frac{2\pi t}{k\log 2}}.
\end{align}
The summand is increasing for $1\leq k \leq (\pi t)^{1/2}/\log 2$ and decreasing for $k \geq
(\pi t)^{1/2}/\log 2$. Moreover, it is bounded from above by $e^{-4(\pi t)^{1/2}}$. We conclude
that for $t \to \infty$,
\begin{align}
\label{e5.78}
&\sum_{k>\lfloor \tfrac14 t^{1/2}\rfloor } e^{-2k\log 2-\frac{2\pi t}{k\log 2}}
\leq 2\,e^{-4(\pi t)^{1/2}} + \int_{[0,\infty)} dk\,e^{-2k\log 2-\frac{2\pi t}{k\log 2}}\nonumber\\
&= 2\,e^{-4(\pi t)^{1/2}} + \frac{(4\pi t)^{1/2}}{\log 2} K_1\big(4(\pi t)^{1/2}\big)
= \frac{\pi^{3/4}}{\sqrt{2}\log 2}\,t^{1/4}\,e^{-4(\pi t)^{1/2}}[1+o(1)],
\end{align}
where we use formula 3.324.1 from \cite{GR} and formula 9.7.2 from \cite{AS}. Putting the
estimates in \eqref{e5.63} and \eqref{e5.73}--\eqref{e5.78} together, we obtain that
\begin{equation*}
%\label{e5.79}
\spadesuit(t) \leq \frac{96\pi^{3/4}\,e^{16C}}{\sqrt{2}\log 2}\,t^{1/4}\,e^{-4(\pi t)^{1/2}}[1+o(1)].
\end{equation*}
This is the desired upper bound in \eqref{e5}.

\medskip\noindent
{\bf 6.} To obtain a lower bound for $\E(D^2_t)$, we consider a
good $k$-square. This square contains a square with the same
centre, parallel sides and area $2^{-2k-2}$. The distance from
this square to $\beta[0,t]$ is bounded from below by $2^{-k-2}$.
Hence
\begin{equation}
\begin{aligned}
\label{e5.80}
\E\left(D^2_t\right)
&\geq \tfrac{1}{16} \sum_{k\in \N} 2^{-2k} \int_{\T^2} dx\,\,
\P_x(S_k \textup{is a good $k$-square})\\
&= \tfrac{1}{16} \sum_{k\in\N} 2^{-2k}\,[p_k(t)-p_{k-1}(t)]
=\tfrac{3}{64} \sum_{k\in\N} 2^{-2k}\,p_k(t),
\end{aligned}
\end{equation}
since $p_0(t)=0$. The following lemma provides a lower bound for the right-hand side of
\eqref{e5.80}.

\begin{lemma}
\label{bbl}
There exists $k_0\in \N$ such that for all $k\ge k_0$,
\begin{equation*}
%\label{e5.81}
p_k(t) \geq \tfrac14 e^{-t\lambda_1(\T^2\backslash S_k)}.
\end{equation*}
\end{lemma}

\begin{proof}
By the eigenfunction expansion in \eqref{a9} we have that
\begin{align*}
%\label{e5.82}
p_k(t) &=\int_{\T^2\backslash S_k} dx \int_{\T^2\backslash S_k} dy\,
\sum_{j\in\N} e^{-t\lambda_j(\T^2\backslash S_k)} \varphi_j(x)\varphi_j(y)\nonumber\\
&\geq e^{-t\lambda_1(\T^2\backslash S_k)}
\left(\int_{\T^2\backslash S_k} dx\,\varphi_1(x)\right)^2.
\end{align*}
By the results of \cite{O}, $\Vert \varphi_1-1\Vert_{L^2(\T^2\backslash S_k)}\ra 0$ as
$k\to\infty$. This implies that $|\int_{\T^2\backslash S_k} dx\,\varphi_1(x)| \geq \tfrac12$ for
$k$ sufficiently large.
\end{proof}

\noindent
Combining \eqref{e5.69}, \eqref{e5.73}, \eqref{e5.80} and Lemma \ref{bbl}, we have that
\begin{equation*}
%\label{e5.83}
\E\left(D^2_t\right)
\geq \tfrac{3}{256} \sum_{\{k\in \N\colon\,k\ge k_0\}}
e^{-2k\log 2-\frac{2\pi t}{k\log 2}-\frac{Ct}{k^2}}.
\end{equation*}
Now let $t$ be such that $\pi t/\log 2>k_0$. Then
\begin{align*}
%\label{e5.84}
\E\left(D^2_t\right)
&\geq \tfrac{3}{256} \sum_{\big\{k\in \N\colon\,k\geq \frac{(\pi t)^{1/2}}{\log 2}\big\}}
e^{-2k\log 2-\frac{2\pi t}{k\log 2}-\frac{Ct}{k^2}}\nonumber\\
&\geq \tfrac{3}{256}\,e^{-C} \sum_{\big\{k\in \N\colon\,k\geq \frac{(\pi t)^{1/2}}{\log 2}\big\}}
e^{-2k\log 2-\frac{2\pi t}{k\log 2}}.
\end{align*}
Because the summand is strictly decreasing in $k$, we can replace the sum over $k$ by
an integral with a minor correction. This gives
\begin{align}
\label{e5.85}
\E\left(D^2_t\right) \geq
\tfrac{3}{256}\, e^{-C} \left(\int_{\frac{(\pi t)^{1/2}}{\log 2}}^{\infty} dk\,
e^{-2k\log 2-\frac{2\pi t}{k\log 2}} - e^{-4(\pi t)^{1/2}}\right).
\end{align}
We have
\begin{align}
\label{e5.86}
&\int_{\frac{(\pi t)^{1/2}}{\log 2}}^{\infty} dk\,e^{-2k\log2-\frac{2\pi t}{k\log 2}}
= \frac{(\pi t)^{1/2}}{\log 2} \int_1^{\infty} dx\,e^{-2(\pi t)^{1/2}(x+\frac{1}{x})}
\geq \frac{(\pi t)^{1/2}}{\log 4} \int_0^{\infty}dx\, e^{-2(\pi t)^{1/2}(x+\frac{1}{x})}\nonumber\\
&= \frac{(\pi t)^{1/2}}{\log 2} K_1\big(4(\pi t)^{1/2}\big)
= \frac{\pi^{3/4}}{2^{3/2}\log 2}t^{1/4}e^{-4(\pi t)^{1/2}}[1+o(1)],
\end{align}
where we use once more formulas 3.324.1 from \cite{GR} and 9.7.2 from \cite{AS}.
Combining \eqref{e5.65}, \eqref{e5.85} and \eqref{e5.86}, we get
\begin{equation*}
%\label{e5.87}
\spadesuit(t) \geq \frac{3\pi^{3/4}\,e^{-C}}{2^{23/2}\log 2}\,t^{1/4}\,e^{-4(\pi t)^{1/2}}[1+o(1)].
\end{equation*}
This is the desired lower bound in \eqref{e5}.
\end{proof}

%%%%%%%%%%%%%%%%%%%%%%%%%%%%%%%%%%%%%%%%%%%%%%%%

\section{Torsional rigidity for $m=3$}
\label{proofthe2}

It is well known that $\beta[0,1]$ has a strictly positive Newton capacity when $m=3$.
In Section~\ref{invCAP} we show that the inverse of the capacity of $\beta[0,1]$ on $\R^3$
has a finite exponential moment. In Section~\ref{linkPDECAP} we show that for every closed
set $K \subset \T^3$ that has a small enough diameter the principal Dirichlet eigenvalue
of $\T^3 \backslash K$ is bounded from below by a constant times the capacity of $K$.
(The same is true for $m \geq 4$, a fact that will be needed in Section~\ref{proofthe3}.) In
Section~\ref{secproofthe2} we use these results to prove Theorem~\ref{the2}.

%%%%

\subsection{Exponential moment of the inverse capacity}
\label{invCAP}

\begin{lemma}
\label{lemcapexp}
Let $m=3$. Then there exists $c > 0$ such that
\begin{equation*}
\E\left(\exp\left[ \frac{c}{\cp(\beta[0,1])}\right]\right) < \infty.
\end {equation*}
\end{lemma}

\begin{proof}
We use the fact that, for any compact set $A\subset\R^3$,
\begin{equation}
\label{caprepr}
\frac{1}{\cp(A)}  = \inf \left[ \int_{\R^3} \int_{\R^3} \frac{\mu(dx)\mu(dy)}{4\pi \left\vert x-y\right\vert}\colon\,
\mu \text{ is a probability measure on } A\right].
\end{equation}
As test probability measure we choose the sojourn measure of $\beta[0,t]$, that is
\begin{equation*}
\mu_{\beta[0,1]}(C) = \int_0^1 1_C( \beta(t))\,dt, \qquad C \subset \R^3,
\end{equation*}
for which
\begin{equation*}
\int_{\R^3} \int_{\R^3} \frac{\mu_{\beta[0,1]}(dx)\mu_{\beta[0,1]}(dy)}{4\pi \left\vert x-y\right\vert}
= \int_0^1 ds \int_0^1 dt\,\,\frac{1}{4\pi \left\vert \beta(s)-\beta(t)\right\vert}.
\end{equation*}
It therefore suffices to prove that
\begin{equation*}
\E_0\left(\exp\left[c\int_0^1 ds \int_0^1 dt\,\,\frac{1}{\left\vert \beta(s)-\beta(t)\right\vert}\right]\right)
< \infty
\end{equation*}
for small enough $c>0$. A proof of this fact is hidden in \cite{DVPolaron}. For the
convenience of the reader we write it out here.

By Cauchy-Schwarz and Jensen, we have that
\begin{align*}
&\E_0\left(\exp\left[c\int_0^1 ds \int_0^1 dt\,\,
\frac{1}{\left\vert \beta(s)-\beta(t)\right\vert}\right]\right)
\leq \E_0\left(\exp\left[2c \int_0^1 ds \int_s^1 dt\,\,
\frac{1}{\left\vert \beta(s)-\beta(t)\right\vert}\right]\right) \nonumber \\
&\leq \E_0\left(\exp\left[2c \int_0^1 ds \int_s^{1+s} dt\,\,
\frac{1}{\left\vert \beta(s)-\beta(t)\right\vert}\right]\right)
\leq \int_0^1 ds\,\,\E_0\left(\exp\left[2c \int_s^{1+s} dt\,\,
\frac{1}{\left\vert \beta(s)-\beta(t)\right\vert}\right]\right) \nonumber \\
&= \E_0\left(\exp\left[2c \int_0^1 dt\,\,\frac{1}{\left\vert\beta(t)\right\vert }\right]\right).
\end{align*}
It therefore suffices to prove that the right-hand side is finite for small enough $c>0$.
Expanding the exponent, we get
\begin{align*}
\E_0\left(\exp\left[2c \int_0^1 dt\,\,\frac{1}{\left\vert\beta(t)\right\vert }\right]\right)
&= \sum_{k\in\N_0} \frac{(2c)^k}{k!}\,\E_0\left(\left[\int_0^1 dt\,\,
\frac{1}{\left\vert\beta(t)\right\vert}\right]^k\right) \nonumber \\
&= \sum_{k\in\N_0} (2c)^k \int_{0 \leq t_1<\cdots<t_k \leq 1}
\E_0\left(\frac{1}{\left\vert\beta(t_1)\right\vert \times\cdots\times
\left\vert\beta(t_k)\right\vert}\right) dt_1 \times \cdots \times dt_k.
\end{align*}
The integrand equals
\begin{equation}
\label{dec1}
\E_0\left(\frac{1}{\left\vert\beta(t_1)\right\vert \times \cdots \times \left\vert\beta(t_{k-1})\right\vert}
\,\E_0\left(\left.  \frac{1}{\left\vert\beta(t_{k-1})+\left[\beta(t_k)-\beta(t_{k-1})\right]
\right\vert}\right\vert \mathcal{F}_{t_{k-1}}\right)\right),
\end{equation}
where $\mathcal{F}_t$ is the sigma-algebra of $\beta$ up to time $t$. However,
\begin{align}
&\E_0\left(\left. \frac{1}{\left\vert \beta(t_{k-1})+\left[\beta(t_k)-\beta(t_{k-1})\right]
\right\vert}\right\vert \mathcal{F}_{t_{k-1}}\right)
= \left. \E_0\left(\frac{1}{\left\vert x+\sqrt{t_{k}-t_{k-1}} \beta(1)\right\vert}\right)
\right\vert_{x=\beta(t_{k-1})} \nonumber \\
&\leq \sup_{x\in\R^3} \E_0\left(\frac{1}{\left\vert x+\sqrt{t_{k}-t_{k-1}}
\beta(1)\right\vert}\right)
\leq \E_0\left(\frac{1}{\left\vert\sqrt{t_{k}-t_{k-1}} \beta(1)\right\vert }\right)
\leq \frac{\gamma}{\sqrt{t_{k}-t_{k-1}}}
\label{dec2}
\end{align}
with $\gamma = \E_0(|\beta(1)|^{-1}) < \infty$, where in the second inequality we use that
$\left\vert x+\beta(1)\right\vert $ is stochastically larger than $\left\vert \beta(1)\right\vert $
for any $x \neq 0$. Iterating \eqref{dec1},\eqref{dec2}, we get
\begin{equation*}
\E_0\left(\frac{1}{\left\vert \beta(t_1)\right\vert \times\cdots\times\left\vert \beta(t_k)\right\vert}\right)
\leq \gamma^k \prod_{i=1}^k \frac{1}{\sqrt{t_{i}-t_{i-1}}},
\end{equation*}
where $t_0=0$. Hence
\begin{align*}
\E_0\left(\exp\left[2c \int_0^1 dt\,\,\frac{1}{\left\vert\beta(t)\right\vert}\right]\right)
&= \sum_{k\in\N_0} (2c)^k \gamma^k \int_{0 \leq t_1<\cdots<t_k \leq 1}
\frac{dt_1}{\sqrt{t_1}} \times \cdots \times \frac{dt_k}{\sqrt{t_k-t_{k-1}}} \nonumber\\
&\leq \sum_{k\in\N_0} (2c)^k \gamma^k \left(\int_0^1 dt\,\,\frac{1}{\sqrt{t}}\right)^k
= \sum_{k\in\N_0} (4c)^k \gamma^k,
\end{align*}
which is finite for $c<1/4\gamma$.
\end{proof}

%%%%

\subsection{Principal Dirichlet eigenvalue and capacity}
\label{linkPDECAP}

\begin{lemma}
\label{eigcap}
Let $m \geq 3$, and let $K$ be a closed subset of $\T^m$ with $\mathrm{diam}(K)
\leq \tfrac12$. Then
\begin{equation}
\label{i0} \lambda_1(\T^m\backslash K) \geq k_m\,\cp(K),
\end{equation}
where
\begin{equation*}
%\label{c}
k_m = \int_0^1 ds\,(4\pi s)^{-m/2}\,e^{-m/4s},
\end{equation*}
and $\cp(K)$ is the Newtonian capacity of $K$ embedded in $\R^m$.
\end{lemma}

\begin{proof}
Since $\mathrm{diam}(K) \leq \tfrac12$, $K$ can be embedded in
$\R^m$ by Lemma \ref{lemlb}(e). We let $K\subset
[-\tfrac12,\tfrac12)^m \subset\R^m$, identify
$[-\tfrac12,\tfrac12)^m$ with $\T^m$, and define $\tilde{K}
\subset \R^m$ by $\tilde{K}=\cup_{k\in \Z^m}\{k+K\}$. Let
$\varphi_1$ be the first eigenfunction on $\T^m\backslash K$ with
Dirichlet boundary conditions on $K$, and let
$\lambda_1(\T^m\backslash K)$ be the corresponding first Dirichlet
eigenvalue. Then
\begin{equation*}
e^{-t\lambda_1(\T^m\backslash K)} \varphi_1(x) =
\int_{\T^m\setminus K} dy\, p_{\T^m\backslash
K}(x,y;t)\,\varphi_1(y).
\end{equation*}
Integrating both sides of this identity over $x \in \T^m\backslash
K$, we get
\begin{equation*}
e^{-t\lambda_1(\T^m\backslash K)} \int_{\T^m\backslash K} dx\,
\varphi_1(x) = \int_{\T^m\backslash K} dx\, \varphi_1(x) -
\int_{\T^m\backslash K} dy\, \P_y(T_K\leq t)\,\varphi_1(y),
\end{equation*}
where $T_K$ is the first hitting time of $K$ by Brownian motion on
$\T^m$. It follows that for any $t>0$,
\begin{align}
\label{i1} \lambda_1(\T^m \backslash K) &= -\frac{1}{t} \,
\log\left(1-\frac{\int_{\T^m \backslash K} dy\,\P_y(T_K\leq
t)\,\varphi_1(y)}
{\int_{\T^m\backslash K} dy\,\varphi_1(y)}\right) \nonumber \\
&\geq \frac{1}{t}\, \frac{\int_{\T^m\backslash K}dy\,\P_y(T_K\leq
t)\,\varphi_1(y)}
{\int_{\T^m\backslash K} dy\,\varphi_1(y)}
\geq \frac{1}{t}\, \inf_{y \in \T^m} \P_y(T_K \leq t),
\end{align}
where we use the inequality $-\log(1-z) \geq z$, $z \in [0,1)$. Let $\tilde{\beta}$
be Brownian motion on $\R^m$, and let $\tilde{T}_{\tilde{K}}$ be the first hitting
time of $\tilde{K}$ by $\tilde{\beta}$. Then
\begin{equation}
\label{i2}
\P_y(T_K \leq t) = \tilde{\P}_y(\tilde{T}_{\tilde{K}} \leq t)
\geq \tilde{\P}_y(\tilde{T}_K \leq t)
\geq \tilde{\P}_y(\tilde{L}_K \leq t),
\end{equation}
where $\tilde{L}_{K}$ is the last exit time from $K$ by
$\tilde{\beta}$. Let $\mu_K$ denote the equilibrium measure on $K$
in $\R^m$. Then (see \cite{PS})
\begin{equation}
\label{i3} \tilde{\P}_y(\tilde{L}_K \leq t) = \int_K \mu_K(dz)
\int_0^t ds\,(4\pi s)^{-m/2}\,e^{-|z-y|^2/4s}.
\end{equation}
By \eqref{i2}--\eqref{i3},
\begin{align}
\label{i4} \inf_{y\in \T^m} \P_y(T_K \leq t) = \inf_{y \in
\big[-\tfrac12,\tfrac12\big)^m} \tilde{\P}_y(\tilde{T}_{\tilde{K}} \leq t)
\geq \inf_{y\in\big[-\tfrac12,\tfrac12\big)^m} \int_K \mu_K(dz) \int_0^t
ds\, (4\pi s)^{-m/2}e^{-|z-y|^2/4s}.
\end{align}
But $|z-y| \leq \sqrt{m}$ for $z\in K$ and $y\in[-\tfrac12,\tfrac12)^m$. Hence the right-hand
side of \eqref{i4} is bounded from below by $\cp(K) \int_0^t ds\, (4\pi s)^{-m/2}\,e^{-m/4s}$.
We now get the claim by choosing $t=1$ in \eqref{i1}.
\end{proof}

We note that if $m=3$ and $K=B_{\epsilon}\subset \T^3$ is a closed ball with radius $\epsilon$,
then $\lambda_1(\T^3\backslash B_{\epsilon}) = \cp(B_{\epsilon})[1+o(1)]$ as $\epsilon
\downarrow 0$ (see \cite{MNP}). In that case, since $k_3=0.0101\dots$, we see that the
constant in \eqref{i0} is off by a large factor.

%%%%%%%%

\subsection{Proof of Theorem~\ref{the2}}
\label{secproofthe2}

\begin{proof}
Write, recalling \eqref{e1},\eqref{e2},\eqref{e3},\eqref{e4}, and using Fubini's theorem,
\begin{equation}
\label{ab}
\spadesuit(t)
= (\E_0 \otimes \tilde\E)\big(\tilde\tau_{\T^3\backslash \beta[0,t]}\big)
= (\E_0 \otimes \tilde\E)\left(\int_0^\infty ds\,
1_{\{\tilde\tau_{\T^3\backslash \beta[0,t]}>s\}}\right)
= (\E_0 \otimes \tilde\E)\left(\int_0^\infty ds\,
1_{\{\tilde\beta[0,s] \cap \beta[0,t] = \emptyset\}}\right),
\end{equation}
where $\tilde\E$ denotes expectation over $\tilde\beta$ with $\tilde\beta(0)$ drawn
uniformly from $\T^3$. By symmetry, we may replace $\E_0 \otimes \tilde\E$ by
$\tilde \E_0 \otimes \E$. The proof comes in 7 Steps. In Steps 1-2 we show that
for a suitable $\eta(t)$, tending to zero as $t\to\infty$,
\begin{equation}
\label{form1}
\spadesuit(t) = [1+o(t)] \int_0^{\eta(t)} ds\,
(\P \otimes \tilde\P)(\beta[0,t] \cap \tilde\beta[0,s] = \emptyset).
\end{equation}
Heuristically, the domain perturbation formula gives that
\begin{equation}
\label{form2}
\begin{aligned}
(\P \otimes \tilde\P)(\beta[0,t] \cap \tilde\beta[0,s] = \emptyset)
&= \tilde\E\left[\exp\left\{-\lambda_1(\T^m \setminus \tilde\beta[0,s])[1+o(1)]\right\}\right]\\
&= \tilde \E\left[\exp\left\{-t\,\cp(\tilde\beta[0,s])[1+o(1)]\right\}\right]\\
&= \tilde\E\left[\exp\left\{-ts^{1/2}\,\cp(\tilde\beta[0,1])[1+o(1)]\right\}\right].
\end{aligned}
\end{equation}
Substituting \eqref{form2} into \eqref{form1} and using the Laplace principle, we get \eqref{e6}.
The details are made precise in Steps 3-7.

\medskip\noindent
{\bf 1.}
Pick $\eta\colon\,(0,\infty) \to (0,\infty)$ such that
\begin{equation}
\label{etacond}
\lim_{t\to\infty} \eta(t) \log t = 0, \qquad
\lim_{t\to\infty} \frac{t\sqrt{\eta(t)}}{\log^2 t} = \infty.
\end{equation}
We begin by showing that the integral over $s \in [\eta(t),\infty)$ decays faster than any
negative power of $t$ and therefore is negligible. Indeed, for any $K(t) \in [\eta(t),\infty)$
we have, by the spectral decomposition in \eqref{a9},
\begin{equation}
\label{lamcapest}
(\tilde\E_0 \otimes \E)\left(\int_{\eta(t)}^{K(t)} ds\,
1_{\{\tilde\beta[0,s] \cap \beta[0,t] = \emptyset\}}\right)
\leq \tilde\E_0\left(\int_{\eta(t)}^{K(t)} ds\,e^{-t \lambda_1(\T^3\backslash\tilde\beta[0,s])}\right).
\end{equation}
By Lemma~\ref{eigcap}, $\lambda_1(\T^3 \backslash A) \geq c_3\,\cp(A)$ for every closed set
$A \subset B_{1/4}(0) \subset \T^3$ (in the lower bound we interpret $A$ as a subset of $\R^3$).
Hence the right-hand side of \eqref{lamcapest} is bounded from above by
\begin{equation}
\label{integralovers}
K(t)\, \tilde\E_0\left(e^{-c_3t\,\cp(\tilde\beta[0,\eta(t)] \cap B_{1/4}(0))}\right),
\end{equation}
where we use that $\cp(\tilde\beta[0,s]) \geq \cp(\tilde\beta[0,\eta(t)])$ for $s \geq \eta(t)$. In
Step 2 we show that $\P_0(\tilde\beta[0,\eta(t)] \subsetneq B_{1/4}(0))$ decays faster than
any negative power of $t$. Hence we may replace $\cp(\tilde\beta[0,\eta(t)] \cap B_{1/4}(0))$
by $\cp(\tilde\beta[0,\eta(t)])$ in \eqref{integralovers} at the cost of a negligible error term
$o(t^{-2})$. Next, we note that $\cp(\tilde\beta[0,\eta(t)])$ is equal to $\sqrt{\eta(t)}\,\cp(\tilde\beta[0,1])$
in distribution. Moreover, since $au+bu^{-1} \geq 2\sqrt{ab}$ for all $a,b,u \in (0,\infty)$, we have,
for any $c>0$,
\begin{equation}
\label{chainineq}
\begin{aligned}
e^{-c_3t \sqrt{\eta(t)}\,\cp(\tilde\beta[0,1])}
&= e^{-c_3t \sqrt{\eta(t)}\,\cp(\tilde\beta[0,1]) - c\,\cp(\tilde\beta[0,1])^{-1}}
e^{c\,\cp(\tilde\beta[0,1])^{-1}}\\
&\leq e^{-2\sqrt{c_3ct\sqrt{\eta(t)}}} e^{c\,\cp(\tilde\beta[0,1])^{-1}}.
\end{aligned}
\end{equation}
By Lemma~\ref{lemcapexp}, we therefore have
\begin{equation*}
\tilde\E_0\left(e^{-c_3t\,\cp(\tilde\beta[0,\eta(t)])}\right)
\leq C\,e^{-2\sqrt{c_3ct\sqrt{\eta(t)}}} + o(t^{-2})
\end{equation*}
for some $C<\infty$ and $c>0$ small enough. Hence \eqref{integralovers} is $O(K(t)^{-1})$
when we pick
\begin{equation*}
K(t) = e^{\sqrt{c_3ct\sqrt{\eta(t)}}}.
\end{equation*}
The second half of \eqref{etacond} ensures that $K(t)$ grows faster than any positive power of $t$,
and so we conclude that the integral in the left-hand side of \eqref{lamcapest} is $o(t^{-2})$.
To estimate
\begin{equation*}
(\tilde\E_0 \otimes \E)\left(\int_{K(t)}^\infty ds\,
1_{\{\tilde\beta[0,s] \cap \beta[0,t] = \emptyset\}}\right)
\end{equation*}
we reverse the roles of $\beta$ and $\tilde\beta$, and do the same estimate using that
$\cp(\beta[0,t]) \geq \cp(\beta[0,\eta(t)])$ for $t \in [\eta(t),\infty)$. This leads to
\begin{equation*}
\begin{aligned}
(\tilde\E_0 \otimes \E)\left(\int_{K(t)}^\infty ds\,
1_{\{\tilde\beta[0,s] \cap \beta[0,t] = \emptyset\}}\right)
&\leq C \int_{K(t)}^\infty ds\, e^{-2\sqrt{c_3cs\sqrt{\eta(t)}}} + o(t^{-2})\\
&= [1+o(1)]\, C\sqrt{\frac{K(t)}{c_3c\sqrt{\eta(t)}}}\,e^{-2\sqrt{c_3cK(t)\sqrt{\eta(t)}}} + o(t^{-2}),
\end{aligned}
\end{equation*}
in which the first term is even much smaller than $o(t^{-2})$.

\medskip\noindent
{\bf 2.}
We next show that the probability that $\tilde\beta$ leaves the ball of radius $\tilde\eta(t) = (M(t)\,
\eta(t) \log t)^{1/2}$ prior to time $\eta(t)$ decays faster than any negative power of $t$ when
$\lim_{t\to\infty} M(t) = \infty$. Indeed, by L\'evy's maximal inequality (Theorem 3.6.5 in \cite{S}),
\begin{equation*}
\begin{aligned}
&\tilde\P_0\big(\exists\,s \in [0,\eta(t)]\colon\,\tilde\beta[0,s]
\notin B_{\tilde\eta(t)}(0)\big) \leq 2\,
\tilde\P_0\big(\tilde\beta(\eta(t)) \notin
B_{\tilde\eta(t)}(0)\big)\\
&=O(\exp\big[-\tfrac18\tilde\eta^2(t)/\eta(t)\big])
= O(\exp[-\tfrac18 M(t) \log t]) = O(t^{-\tfrac18 M(t)}) = o(t^{-2}).
\end{aligned}
\end{equation*}
Hence, with a negligible error we may restrict the expectation in the right-hand side of
\eqref{ab} to the event
\begin{equation}
\label{Etdef}
\cE_t = \{\tilde\beta[0,\eta(t)] \subset B_{\tilde\eta(t)}(0)\}.
\end{equation}
The first half of \eqref{etacond} guarantees that $\lim_{t\to\infty} \tilde\eta(t) = 0$ for
some choice of $M(t)$ with $\lim_{t\to\infty} M(t)=\infty$.

\medskip\noindent
{\bf 3.}
We proceed by estimating the number of excursions between the boundaries of two concentric
balls. Fix $0<\delta <\tfrac18$, and consider the successive excursions of $\beta$ between
the boundaries of the balls $B_{1/4}(0)$ and $B_\delta(0)$, i.e., put $\sigma_0 = \inf\{u
\geq 0\colon\,\beta(u) \in \partial B_{1/4}(0)\}$ and, for $k\in\N$,
\begin{equation*}
\begin{aligned}
\bar\sigma_k &= \inf\{u \geq \sigma_{k-1}\colon\,\beta(u) \in \partial B_\delta(0)\},\\
\sigma_k &= \inf\{u \geq \bar\sigma_k\colon\,\beta(u) \in \partial B_{1/4}(0)\}.
\end{aligned}
\end{equation*}
For $k\in\N$, let $\beta_k = \beta([\sigma_{k-1},\sigma_k])$ denote the $k$-th excursion from
$\partial B_{1/4}(0)$ to $\partial B_\delta(0)$ and back. Let $\bar X_k = \beta(\bar\sigma_k)$
denote the location where this excursion first hits $\partial B_\delta(0)$.  Clearly, under the law
$\P$, $(\bar\sigma_k-\sigma_{k-1},\sigma_k-\bar\sigma_k,\bar X_k)_{k\in\N}$ is a uniformly
ergodic Markov chain on $(0,\infty)^2 \times \T^3$. Let
\begin{equation}
\label{Ndeltatdef}
N_\delta(t) = \sup\{k\in\N\colon\,\sigma_k \leq t\}
\end{equation}
be the number of completed excursions prior to time $t$. By the renewal theorem, we have
\begin{equation*}
%\label{mean}
\lim_{t\to\infty} t^{-1}\E(N_\delta(t)) = \frac{1}{e_\delta+e'_\delta},
\qquad e_\delta = \E(\bar\sigma_1-\sigma_0), \quad e'_\delta = \E(\sigma_1-\bar\sigma_1).
\end{equation*}
Moreover, for every $\delta'>0$ there exists a $C_\delta(\delta')>0$ such that
\begin{equation}
\label{dev}
\P\Big(t^{-1}|N_\delta(t)-(e_\delta+e'_\delta)^{-1}| \geq \delta'\Big)
\leq e^{-C_\delta(\delta')t}, \qquad t \geq 0,
\end{equation}
where we have used the fact that $\sigma_k$ and $\bar\sigma_k$ have finite
exponential moments.

\medskip\noindent
{\bf 4.}
We proceed by estimating the probability that an excursion between the boundaries of
two concentric balls hits $\tilde\beta[0,s]$. Fix $\tilde\beta[0,\eta(t)] \subset B_{\tilde\eta(t)}
(0)$. For $s \in [0,\eta(t)]$ and $N\in\N$, the probability that the first $N$ excursions do
not hit $\tilde\beta[0,s]$ equals
\begin{equation}
\label{PiNdef}
\Pi\big(N;\tilde\beta[0,s]\big)
= \E\left(\prod_{k=1}^N 1_{\{\tilde\beta[0,s] \cap \beta_k = \emptyset\}}\right)
= \E\left(\E\left( \prod_{k=1}^N \big[1-p\big(\bar X_k,\bar X_{k+1};\tilde\beta[0,s]\big)\big]
~\Bigg|~ \cF_{N+1} \right)\right),
\end{equation}
where $\cF_{N+1}$ is the sigma-algebra generated by $\bar X_k$, $1 \leq k \leq N+1$,
and
\begin{equation*}
p\big(x,y;\tilde\beta[0,s]\big) = \P_x^y\left(\sigma_{\tilde\beta[0,s]}<\infty\right),
\qquad x,y \in \partial B_\delta(0),
\end{equation*}
is the probability that a Brownian motion, starting from $x \in \partial B_\delta(0),$ and
conditioned to re-enter $B_\delta(0)$ at $y \in \partial B_\delta(0)$ after it has exited
$B_{1/4}(0)$, hits $\tilde\beta[0,s]$. The following lemma gives a sharp estimate of
$p(x,y;\tilde\beta[0,s])$.

\begin{lemma}
\label{hitest}
If $\tilde\eta(t) \leq \tfrac12\delta$ and $\tilde\beta[0,\eta(t)] \subset B_{\tilde\eta(t)}(0)$,
then
\begin{equation}
\label{est2} p\big(x,y;\tilde\beta[0,s]\big) = [1+O(\delta)]
\left\{(\kappa_3\delta)^{-1} \cp(\tilde\beta[0,s]) + O(\delta^{-2})\,\tilde\eta^2(t)\right\},
\qquad \delta\downarrow 0,
\end{equation}
for all $x,y \in \partial B_\delta(0)$ and $s \in [0,\eta(t)]$.
\end{lemma}

\begin{proof}
We begin by showing that if $\tilde\eta(t) \leq \tfrac12\delta$, then
\begin{equation}
\label{est1}
\big|\P_x\big(\sigma_{\tilde\beta[0,s]}<\infty\big) - (\kappa_3\delta)^{-1}\cp(\tilde\beta[0,s])\big|
\leq 2\delta^{-2} \tilde\eta^2(t)
\end{equation}
for all $x \in \partial B_\delta(0)$ and $\tilde\beta[0,s] \subset B_{\tilde\eta(t)}(0)$.
Indeed, for any compact set $K \subset \R^3$, we have
\begin{equation}
\label{hatPapprox} \cp(K) = \int_K \mu_K(dy), \qquad
\P_x(\sigma_K<\infty) = \int_K \frac{\mu_K(dy)}{\kappa_3 |x-y|}, \qquad x \in K,
\end{equation}
where $\mu_K$ is the equilibrium measure on $K$ (see \cite{S1}, \cite{PS}, \cite{Sz}).
If $|x|=\delta$ and $|y| \leq \tfrac12\delta$, then $||x-y|^{-1}-|x|^{-1}| \leq 2\delta^{-2}|y|$.
Hence \eqref{hatPapprox} yields the estimate $|\P_x(\sigma_K<\infty) - (\kappa_3\delta)^{-1}
\cp(K)| \leq 2\kappa_3^{-1}\delta^{-2}\tilde\eta(t) \cp(K)$, provided $K \subset B_{\tilde\eta(t)}(0)$.
In that case $\cp(K) \leq \cp(B_{\tilde\eta(t)}(0)) = \kappa_3 \tilde\eta(t)$, and the claim
in \eqref{est1} follows. Furthermore, since $\P_a(\sigma_{B_\delta (0)}<\infty)
= \kappa_3(4\delta)$ for all $a \in B_{1/4}(0)$, we have
\begin{equation*}
0 \leq \P_x\big(\sigma_{\tilde\beta[0,s]}<\infty\big) -
\inf_{y \in \partial B_\delta(0)}p\big(x,y;\tilde\beta[0,s]\big) \leq
\kappa_3 (4\delta) \sup_{y,z \in \partial B_\delta(0)} p\big(y,z;\tilde\beta[0,s]\big),
\, x \in \partial B_\delta(0).
\end{equation*}
Hence \eqref{est1} implies \eqref{est2}.
\end{proof}

\medskip\noindent
{\bf 5.}
We proceed by estimating the integral over $s \in [0,\eta(t)]$ that supplements
\eqref{lamcapest}. Recalling \eqref{Ndeltatdef}, we have
\begin{equation*}
1_{\{\tilde\beta[0,s] \cap \beta[0,\sigma_0] = \emptyset\}}
\,\prod_{k=1}^{N_\delta(t)+1} 1_{\{\tilde\beta[0,s] \cap \beta_k = \emptyset\}}
\leq 1_{\{\tilde\beta[0,s] \cap \beta[0,t] = \emptyset\}}
\leq \prod_{k=1}^{N_\delta(t)} 1_{\{\tilde\beta[0,s] \cap \beta_k = \emptyset\}}.
\end{equation*}
In terms of the probability defined in \eqref{PiNdef}, and with the help of the large deviation
estimate in \eqref{dev}, this sandwich gives us, on the event $\cE_t$,
\begin{equation*}
\E\left(\int_0^{\eta(t)} ds\,1_{\{\tilde\beta[0,s] \cap \beta[0,t] = \emptyset\}}\right)
= O\left(\eta(t)\,e^{-C_\delta(\delta')t}\right)
+ [1+o_t(1)]\, \Pi\Big([1+o_t(1)](e_\delta+e'_\delta)^{-1}t;\tilde\beta[0,s]\Big),
\end{equation*}
where the error terms $o_t(1)$ tend to zero as $t\to\infty$ (here we use that $\lim_{t\to\infty}
\P(B_{\tilde\eta(t)}(0) \cap \beta[0,\sigma_0] = \emptyset)=1$).

\medskip\noindent
{\bf 6.}
Combining the estimates in Steps 1--5, and using that $\cp(\tilde\beta[0,s])$ equals
$\cp(\tilde\beta[0,1])\sqrt{s}$ in distribution under $\tilde\P_0$, we get
\begin{equation}
\label{lim1}
\begin{aligned}
\spadesuit(t)
&= o(t^{-2}) + [1+o_t(1)]\,\tilde\E_0 \left(\int_0^{\eta(t)} ds\,e^{-A_\delta(t)\sqrt{s}}\right)\\
&= o(t^{-2}) + [1+o_t(1)]\,\tilde\E_0\left(\frac{2}{A_\delta(t)^2}\,
\left\{1+e^{-A_\delta(t)\sqrt{\eta(t)}}\,[A_\delta(t)\sqrt{\eta(t)}-1]\right\}\right),
\end{aligned}
\end{equation}
with
\begin{equation}
\label{lim2}
t^{-1}A_\delta(t) = [1+O(\delta)]\,[1+o_t(1)]\,(e_\delta+e'_\delta)^{-1}\,(\kappa_3\delta)^{-1}\cp(\tilde\beta[0,1]),
\qquad t \to \infty.
\end{equation}
The term between braces in \eqref{lim1} is bounded, and tends to $1$ in $\tilde\P_0$-probability as
$t\to\infty$ because of the first half of \eqref{etacond}. Therefore \eqref{lim1},\eqref{lim2}
lead us, for fixed $\delta$, to
\begin{equation*}
\lim_{t\to\infty} t^2\spadesuit(t) = \lim_{t\to\infty} t^2\,\tilde\E_0\left(\frac{2}{A_\delta(t)^2}\right)
= [1+O(\delta)]\,2(\kappa_3\delta)^2(e_\delta+e'_\delta)^2\,\tilde\E_0\left(\frac{1}{\cp(\tilde\beta[0,1])^2}\right),
\end{equation*}
where we have used Lemma~\ref{lemcapexp}. The latter also implies that the expectation in the right-hand side is
finite.

\medskip\noindent
{\bf 7.}
Finally, letting $\delta \downarrow 0$ and using that
\begin{equation}
\label{delcaplim}
\lim_{\delta \downarrow 0} \delta e_\delta = 1/\kappa_3, \qquad
\lim_{\delta \downarrow 0} e'_\delta =
\E_0(\tau_{B_{1/4}(0)})<\infty,
\end{equation}
we arrive at
\begin{equation*}
\lim_{t\to\infty} t^2\spadesuit(t) =
2\,\tilde\E_0\left(\frac{1}{\cp(\tilde\beta[0,1])^2}\right).
\end{equation*}
This proves the claim in \eqref{e6}.
\end{proof}

%%%%%%%%%%%%%%%%%%%%%%%%%%%%%%%%

\section{Torsional rigidity for $m \geq 4$}
\label{proofthe3}

The same estimates as in the proof of Theorem~\ref{the2} for $m=3$ in
Section~\ref{secproofthe2}  can be used to prove Theorem~\ref{the3} for
$m \geq 4$ after we replace $\tilde\beta[0,s]$ by $\tilde W_{r(t)}[0,s]$.
The details are explained in Sections~\ref{proofthe35ormore}--\ref{proofthe34}.

%%%%%%%%%%%%%

\subsection{Proof of Theorem~\ref{the3} for $m \geq 5$}
\label{proofthe35ormore}

\begin{proof}
In the proof we assume that
\begin{equation}
\label{rcond}
\lim_{t\to\infty} t^{1/(m-2)} r(t) = 0, \qquad  \lim_{t\to\infty} \frac{t}{\log^3 t}\,r(t)^{m-4} = \infty.
\end{equation}

\medskip\noindent
{\bf 1-2.}
The estimates in Steps 1--2 are sharp enough to produce a negligible error term $o(t^{-2/(m-2)})$
when \eqref{etacond} is replaced by
\begin{equation}
\label{etacondalt}
\lim_{t\to\infty} \eta(t) \log t = 0, \qquad
\lim_{t\to\infty} \frac{t\,r(t)^{m-4}\,\eta(t)}{\log^2 t} = \infty,
\end{equation}
where we note that by the second half of \eqref{rcond} there exists a choice of $\eta(t)$
satisfying \eqref{etacondalt}. Indeed, the analogues of \eqref{lamcapest}--\eqref{integralovers}
give (recall that Lemma~\ref{eigcap} also holds for $m \geq 4$)
\begin{equation}
\label{integralovers*}
(\tilde\E_0 \otimes \E)\left(\int_{\eta(t)}^{K(t)} ds\,
1_{\{\tilde{W}_{r(t)}[0,s] \cap \beta[0,t] = \emptyset\}}\right)
\leq K(t)\, \tilde\E_0\left(e^{-c_mt\,\cp(\tilde{W}_{r(t)}[0,\eta(t)] \cap B_{1/4}(0))}\right),
\end{equation}
where we use that $\cp(\tilde{W}_{r(t)}[0,s]) \geq \cp(\tilde{W}_{r(t)}[0,\eta(t)])$ for $s \geq
\eta(t)$. The estimate in Step 2 shows that, because of the first half of \eqref{etacondalt},
$\P_0(\cE_t^c)$ with $\cE_t$ defined in \eqref{Etdef} decays faster than any negative power
of $t$, so that we can remove the intersection with $B_{1/4}(0)$ at the expense of a negligible
error term. Since $t\,\cp(\tilde{W}_{r(t)}[0,\eta(t)])$ equals $t\,r(t)^{m-2}\cp(\tilde{W}_1[0,\eta(t)/r(t)^2])$
in distribution under $\tilde\P_0$, we obtain that
\begin{equation*}
\tilde\E_0\left(e^{-c_mt\,\cp(\tilde{W}_{r(t)}[0,\eta(t)])}\right)
= \tilde\E_0\left(e^{-c_m\,t\,r(t)^{m-2}\,\cp(\tilde{W}_1[0,\eta(t)/r(t)^2])}\right).
\end{equation*}
Via an estimate similar as in \eqref{chainineq} with $c$ replaced by $c\eta(t)/r(t)^2$, we obtain,
with the help of Lemma~\ref{lemcapest*} below (which is the analogue of Lemma~\ref{lemcapexp}
and is proved in Section~\ref{invCAPt}),
\begin{equation*}
\tilde{\E}_0\left(e^{-c_m\,t\,r(t)^{m-2}\,\cp(\tilde{W}_1[0,\eta(t)/r(t)^2])}\right)
\leq C\,e^{-2\sqrt{c_mc\, t\,r(t)^{m-2}\,\eta(t)/r(t)^2}} + o(t^{-2/(m-2)}).
\end{equation*}
Hence the right-hand side of \eqref{integralovers*} is $O(K(t)^{-1})$ when we pick
\begin{equation*}
K(t) = e^{\sqrt{c_mc\, t\,r(t)^{m-2}\,\eta(t)/r(t)^2}}.
\end{equation*}
The second half of \eqref{etacondalt} ensures that $K(t)$ grows faster than any positive power
of $t$, and so \eqref{integralovers*} is negligible. The contribution
\begin{equation*}
(\tilde\E_0 \otimes \E)\left(\int_{K(t)}^\infty ds\,
1_{\{\tilde{W}_{r(t)}[0,s] \cap \beta[0,t] = \emptyset\}}\right)
\end{equation*}
can again be estimated in a similar way by reversing the roles of $\beta$ and $\tilde\beta$.
This leads to a term that is even much smaller.

\medskip\noindent
{\bf 3-5.}
Step 3 is unaltered. In Step 4 the term $\delta^{-1}$ is to be replaced by $\delta^{-(m-2)}$,
because in \eqref{hatPapprox} the term $1/\kappa_3|x-y|$ is to be replaced by $1/\kappa_m
|x-y|^{m-2}$. Step 5 is unaltered.

\medskip\noindent
{\bf 6-7.}
In Step 6 we use that $\cp(\tilde W_{r(t)}[0,s])$ equals $s^{(m-2)/2} \cp(\tilde W_{r(t)/\sqrt{s}}[0,1])$
in distribution under $\tilde\P_0$. This gives
\begin{equation}
\label{spadert}
\spadesuit_{r(t)}(t) = I_1(t) + o(t^{-2/(m-2)}),
\end{equation}
where
\begin{equation*}
%\label{I1}
I_1(t) = \tilde\E_0\left(\int_0^{\eta(t)} ds\,e^{-A_\delta(t,s)\,s^{(m-2)/2}}\right),
\end{equation*}
with
\begin{equation*}
A_\delta(t,s) = [1+O(\delta)]\,[1+o_t(1)]\,(e_\delta+e'_\delta)^{-1}\,t\,\delta^{-(m-2)}\,
\cp(\tilde W_{r(t)/\sqrt{s}}[0,1]).
\end{equation*}
With the change of variable $u=t^{1/(m-2)}\sqrt{s}$, the integral becomes
\begin{equation}
\label{I1I2}
I_1(t) = t^{-2/(m-2)} I_2(t),
\end{equation}
where
\begin{equation}
\label{I2}
I_2(t) =  \tilde\E_0\left(2\int_0^{t^{1/(m-2)}\sqrt{\eta(t)}} du\,u\,e^{-A'_\delta(t,u)\,u^{m-2}}\right)
\end{equation}
with (recall \eqref{e12})
\begin{equation}
\label{I2*}
A'_\delta(t,u) = [1+O(\delta)]\,[1+o_t(1)]\,(e_\delta+e'_\delta)^{-1}\,\delta^{-(m-2)}\,
\cp(\tilde W_{\epsilon(t)/u}[0,1]),
\end{equation}
and where $\eps(t) = t^{1/(m-2)}r(t)$. Now, \eqref{capscalapr} tells us that
\begin{equation*}
\cp(\tilde W_{\epsilon(t)/u}[0,1]) = [1+o(1)]\,u^{-(m-4)}\cp(\tilde W_{\epsilon(t)}[0,1])
\end{equation*}
in $\P_0$-probability as $t\to\infty$ for every $u \in (0,\infty)$ and $m \geq 5$, where we use that $\eps(t) = o(1)$ by the first half of
\eqref{rcond}. Therefore with the help of \eqref{etacondalt} and dominated convergence, we find
that
\begin{equation}
\label{I2lim}
I_2(t) = [1+o(1)]\, \tilde\E_0\left(2 \int_0^\infty du\,u\,e^{-A''_\delta(t)\,u^2}\right)
= [1+o(1)]\,\tilde\E_0\left(\frac{1}{A''_\delta(t)}\right), \qquad t \to \infty,
\end{equation}
with
\begin{equation}
\label{I2lim*}
A''_\delta(t) = [1+O(\delta)]\,[1+o_t(1)]\,(e_\delta+e'_\delta)^{-1}\,\delta^{-(m-2)}\,
\cp(\tilde W_{\epsilon(t)}[0,1]).
\end{equation}
In Step 7 the first line in \eqref{delcaplim} is replaced by the statement that $\lim_{\delta
\downarrow 0}\delta^{m-2} e_{\delta}=1/\kappa_m$. Combining \eqref{spadert}, \eqref{I1I2}
and \eqref{I2lim}, and letting $\delta\downarrow 0$, we get the scaling in \eqref{e13}.
\end{proof}

%%%%%%%%%%%%%%%%%

\subsection{Proof of Theorem~\ref{the3} for $m=4$}
\label{proofthe34}

\begin{proof}
In the proof we assume that
\begin{equation}
\label{rcond*}
\lim_{t\to\infty} t^{1/(m-2)}r(t) = 0, \qquad \lim_{t\to\infty} \frac{t}{\log^3 t}\,\frac{1}{\log(1/r(t))} = \infty.
\end{equation}

\medskip\noindent
{\bf 1-2.}
The estimates in Steps 1--2 are sharp enough to produce a negligible error term $o(t^{-2/(m-2)})$
when \eqref{etacond} is replaced by
\begin{equation}
\label{etacondalt*}
\lim_{t\to\infty} \eta(t) \log t = 0, \qquad
\lim_{t\to\infty} \frac{t}{\log^2 t}\,\frac{\eta(t)}{\log(\eta(t)/r(t)^2)} = \infty,
\end{equation}
where we note that by the second half of  \eqref{rcond*} there exists a choice of $\eta(t)$ satisfying
\eqref{etacondalt*}. The estimate uses \eqref{chainineq} with $c$ replaced by
$c(\eta(t)/r(t)^2)/\log(\eta(t)/r(t)^2)$, and also Lemma~\ref{lemcapest*} below (which is
the analogue of Lemma~\ref{lemcapexp} and is proved in Section~\ref{invCAPt}).

\medskip\noindent
{\bf 3-5.}
These steps are unaltered.

\medskip\noindent
{\bf 6-7.}
These steps are unaltered: \eqref{defscalcap},\ref{capscalm4} tell us that
\begin{equation*}
\cp(\tilde W_{\epsilon(t)/u}[0,1])= [1+o(1)]\,\cp(\tilde W_{\epsilon(t)}[0,1])
\quad \text{in $\P_0$-probability as $t\to\infty$}
\end{equation*}
for every $u \in (0,\infty)$, where we use that $\eps(t) = t^{1/(m-2)}r(t) = o(1)$ by the first half
of \eqref{rcond*}. This is used in \eqref{I2},\eqref{I2*} to get \eqref{I2lim},\eqref{I2lim*} with
$m=4$.
\end{proof}

%%%%%%%%%%%%%%%%%%%%%%%%%%%%%%%%%%%%%%%%

\section{Proof of Theorem~\ref{the4}}
\label{extra}

\begin{proof}
By a direct calculation via the Fourier transform, we have that the Dirichlet heat kernel
on $\T^m$ is given by (recall the notation in Section~\ref{background})
\begin{equation*}
%\label{a20}
 p_{\T^m}(x,y;s) = (4\pi s)^{-m/2} \sum_{\lambda\in
(2\pi\Z)^m} e^{-|x-y-\lambda|^2/4s},
\end{equation*}
where $|x-y-\lambda|=d(x-y,\lambda)$. It follows that
\begin{equation}
\label{a21}
 p_{\T^m}(x,x;s) = (4\pi s)^{-m/2}
\sum_{\lambda\in\Z^m} e^{-\pi^2|\lambda|^2/s}.
\end{equation}
By translation invariance, $p_{\T^m}(x,x;s)$ is independent of
$x$, and we will denote it by $\pi(s)$. By the eigenfunction
expansion in \eqref{a9} with $M=\T^m$ and $\O=\cB(t)=\T^m
\backslash\beta[0,t]$, and by the monotonicity of the Dirichlet
heat kernel, we have for $s>0$,
\begin{equation*}
%\label{a211}
e^{-s\lambda_t}\varphi(x)^2 \leq p_{\cB(t)}(x,x;s)\le\pi(s),
\end{equation*}
where we abbreviate $\lambda_t = \lambda_1(\cB(t))$ as in \eqref{abbrdef}.
Taking the supremum over $x$, we obtain
\begin{equation*}
%\label{a212}
 \lVert \varphi_1
\rVert_{L^{\infty}(\cB(t))}^{-2} \geq
\pi(s)^{-1} e^{-s\lambda_t}.
\end{equation*}
By Lemma \ref{lemlb}(b) we have, for $s>0$,
\begin{equation*}
%\label{a213}
 \cT(\cB(t)) \geq
\lambda_t^{-1}\,\pi(s)^{-1} e^{-s\lambda_t}.
\end{equation*}
Since $q\mapsto q^{-1}e^{-sq}$ is convex for every $s>0$, Jensen gives that
\begin{equation}
\label{a214}
\spadesuit(t) \geq \pi(s)^{-1}\E_0(\lambda_t)^{-1} e^{-s\E_0(\lambda_t)}.
\end{equation}
For $s=1$ this reads
\begin{equation}
\label{a215}
\E_0(\lambda_t)\,e^{\E_0(\lambda_t)} \geq \pi(1)^{-1} \spadesuit(t)^{-1}.
\end{equation}
Since the right-hand side of \eqref{a215} increases to infinity as
$t\to\infty$, there exists  $t_0<\infty$ such that
$\E_0(\lambda_t) \geq1$ for $t \geq t_0$. We now put
\begin{equation*}
%\label{a216}
s_t=\E_0(\lambda_t)^{-1}
\end{equation*}
and note that $s_t \leq 1$ for $t \geq t_0$. By \eqref{a21} and \eqref{a214},
we find that, for $t \geq t_0$,
\begin{align}
%\label{a217}
\spadesuit(t) &\geq e^{-1}\pi(s_t)^{-1} s_t
=(4\pi)^{m/2}e^{-1}\,s_t^{(2+m)/2} \left(\sum_{\lambda\in\Z}
e^{-\pi^2|\lambda|^2/s_t}\right)^{-m}\nonumber \\ & \geq
(4\pi)^{m/2}e^{-1}\,s_t^{(2+m)/2} \left(\sum_{\lambda\in\Z}
e^{-\pi^2|\lambda|^2}\right)^{-2} \geq s_t^{(2+m)/2}.
\end{align}
We conclude that, for $t \geq t_0$,
\begin{equation*}
%\label{a218}
 \E_0(\lambda_t) \geq \spadesuit(t)^{-2/(m+2)}.
\end{equation*}
\end{proof}

%%%%%%%%%%%%% SECTION 7 %%%%%%%%%%%%%%%%%%%%%%

\section{Capacity of Wiener sausage for $m \geq 4$}
\label{scalCAP}

In Section~\ref{invCAPt} we derive the analogue of Lemma~\ref{lemcapexp}, showing
that the inverse of $\cC(t)$ for $m \geq 4$ defined in \eqref{defscalcap} has a finite
exponential moment uniformly in $t \geq 2$. In Section~\ref{proofcapscal} we prove
\eqref{defscalcap}--\eqref{capscalm4} for $m \geq 5$.

%%%%

\subsection{Exponential moment of the inverse capacity}
\label{invCAPt}

\begin{lemma}
\label{lemcapest*}
Let $m \geq 4$. Then there exists a $c>0$ such that
\begin{equation*}
\sup\limits_{t \geq 2} \E_0\left(\exp\left[\frac{c}{\cC(t)}\right]\right) < \infty.
\end{equation*}
\end{lemma}

\begin{proof}
The proof is similar to that of Lemma~\ref{lemcapexp}. For any compact set $A \subset\R^m$,
we use the representation (compare with \eqref{caprepr})
\begin{equation}
\label{caprepr*}
\frac{1}{\cp(A)} = \inf \left[ \int_{\R^m} \int_{\R^m} \frac{\mu(dx) \mu(dy)}
{\kappa_m\left\vert x-y\right\vert^{m-2}}\colon\, \mu \text{ is a probability
measure on }A\right].
\end{equation}
As test probability measure we choose the sojourn measure of $W_1[0,t]$,
namely,
\begin{equation}
\label{muWiener}
\mu_{W_1[0,t]} = \frac{1}{t} \int_0^t \nu_{\beta(s)}\,ds \quad \text{ with } \quad
\nu_z(dx) = \frac{1}{\omega_m}\,1_{B_1(z)}(x)\,dx, \quad z \in \R^m,
\end{equation}
where $\omega_m = |B_1(0)|$. Since $\mu$ has support in $W_1[0,t]$, we have
\begin{equation*}
\frac{1}{\cp(W_1[0,t])} \leq \frac{1}{\kappa_m\omega_m^2 t^2 }
\int_0^t du \int_0^t dv\, \int_{B_1(0)} dx \int_{B_1(0)} dy\,\,
\frac{1}{\left\vert\beta(u)+x-\beta(v)-y\right\vert^{m-2}}.
\end{equation*}
Moreover, there exists $C=C(m)>0$ such that for all $u$ and $v$,
\begin{equation*}
\int_{B_1(0)} dx \int_{B_1(0)} dy \,\,\frac{1}{\left\vert\beta(u)+x-\beta(v)-y\right\vert^{m-2}}
\leq \frac{C}{\left\vert\beta(u)-\beta(v)\right\vert^{m-2} \vee 1}.
\end{equation*}

We first prove the claim for $m \geq 5$. Let $\bar{c}=c\,C/\kappa_m\omega_m^2$. We
have that
\begin{equation}
\label{eqest1}
\begin{aligned}
\exp\left[\frac{c}{\cC(t)}\right]
&\leq \exp\left[\frac{\bar{c}}{t} \int_0^t du
\int_0^t dv\,\,\frac{1}{\left\vert\beta(u)-\beta(v)\right\vert^{m-2} \vee 1}\right]\\
&\leq \frac{1}{t} \int_0^t du\,\exp\left[\bar{c}
\int_0^t dv\,\,\frac{1}{\left\vert\beta(u)-\beta(v)\right\vert^{m-2} \vee 1}\right]\\
&\leq \frac{1}{t} \int_0^t du\,\exp\left[\bar{c}
\int_\R dv\,\,\frac{1}{\left\vert\beta(u)-\beta(v)\right\vert^{m-2} \vee 1}\right] .
\end{aligned}
\end{equation}
Taking the expectation and using the translation invariance of Brownian motion,
we obtain the $t$-independent bound
\begin{equation}
\label{eqest2}
\begin{aligned}
\E_0\left(\exp\left[\frac{c}{\cC(t)}\right]\right)
&\leq \E_0\left(\exp\left[\bar{c}
\int_\R dv\,\,\frac{1}{\left\vert\beta(v)\right\vert^{m-2} \vee 1}\right]\right)\\
&\leq \E_0\left(\exp\left[2\bar{c}
\int_0^\infty dv\,\, \frac{1}{\left\vert\beta(v)\right\vert^{m-2} \vee 1}\right]\right),
\end{aligned}
\end{equation}
and so it remains to show that the right-hand side is finite for $c$ small enough.
Arguing in the same way as in the proof of Lemma~\ref{lemcapexp}, we obtain
\begin{equation}
\label{eqest3}
\begin{aligned}
&\E_0\left(\exp\left[2\bar{c}
\int_0^\infty dv\,\,\frac{1}{\left\vert\beta(v)\right\vert^{m-2} \vee 1}\right]\right) \\
&\leq \sum_{k \in \N_0} (2\bar{c})^k
\,\E_0\left(\int_{0 \leq v_1 < \cdots < v_k < \infty} \prod_{i=1}^k
\frac{dv_i}{\left\vert\beta(v_i)\right\vert^{m-2} \vee 1}\right) \\
&\leq \sum_{k \in \N_0} (2\bar{c})^k
\left[\int_0^\infty dv\,\, \E_0\left(\frac{1}{\left\vert\beta(v)\right\vert^{m-2} \vee 1}\right)\right]^k.
\end{aligned}
\end{equation}
Therefore it remains to prove the finiteness of the integral. That that end, we estimate
\begin{equation}
\label{vintest}
\begin{aligned}
\int_0^\infty dv\,\,\E_0\left(\frac{1}{\left\vert\beta(v)\right\vert^{m-2} \vee 1}\right)
&\leq 1 + \int_1^\infty dv\,\, \E_0\left(\left\vert\beta(v)\right\vert^{-(m-2)} \wedge 1\right)\\
&\leq 1 + \int_1^\infty dv\,\, \E_0\left(\left\vert\beta(v)\right\vert^{-(m-2)}\right)\\
&= 1+ \E_0\left(\left\vert\beta(1)\right\vert^{-(m-2)}\right) \int_1^\infty dv\, v^{-(m-2)/2} < \infty,
\end{aligned}
\end{equation}
where the last inequality holds because $m\geq 5$.

We finish by proving the claim for $m = 4$. Let $\bar{c} = c\,C/\kappa_4\omega_4^2$,
and replace \eqref{eqest1} by
\begin{equation*}
%\label{eqest1*}
\exp\left[\frac{c}{\cC(t)}\right]
\leq \frac{1}{t} \int_0^t du\,\exp\left[\frac{\bar{c}}{\log t}
\int_{u-t}^{u+t} dv\,\,\frac{1}{\left\vert\beta(u)-\beta(v)\right\vert^2 \vee 1}\right],
\end{equation*}
and \eqref{eqest2} by
\begin{equation*}
%\label{eqest2*}
\E_0\left(\exp\left[\frac{c}{\cC(t)}\right]\right)
\leq \E_0\left(\exp\left[\frac{2\bar{c}}{\log t}
\int_0^t dv\,\, \frac{1}{\left\vert\beta(v)\right\vert^2 \vee 1}\right]\right),
\end{equation*}
and \eqref{eqest3} by
\begin{equation*}
%\label{eqest3*}
\E_0\left(\exp\left[\frac{2\bar{c}}{\log t}
\int_0^t dv\,\,\frac{1}{\left\vert\beta(v)\right\vert^2 \vee 1}\right]\right) \\
\leq \sum_{k \in \N_0} \left(\frac{2\bar{c}}{\log t}\right)^k
\left[\int_0^t dv\,\, \E_0\left(\frac{1}{\left\vert\beta(v)\right\vert^2 \vee 1}\right)\right]^k,
\end{equation*}
and \eqref{vintest} by
\begin{equation*}
%\label{vintest*}
\int_0^t dv\,\,\E_0\left(\frac{1}{\left\vert\beta(v)\right\vert^2 \vee 1}\right)
= 1+ \E_0\left(\left\vert\beta(1)\right\vert^{-2}\right) \int_1^t dv\, v^{-1}
\leq c'\log t,
\end{equation*}
for some $c' \in (0,\infty)$.
\end{proof}

%%%%%%%%%%%%%%%%

\subsection{Scaling of the capacity}
\label{proofcapscal}

We close by settling \eqref{capscalm4}. The proof for $m \geq 5$ is easy and uses subadditivity.
The proof for $m =4$ is much more complicated and is given in \cite{ASS}.

Note that capacity is subadditive: $\cp(W_1[0,s+t]) \leq \cp(W_1[0,s]) + \cp(W_1[s,s+t])$
for all $s,t \geq 0$. Hence, Kingman's subadditive ergodic theorem yields that
\begin{equation*}
\lim_{t\to\infty} t^{-1} \cp(W_1[0,t]) = \bar{c}_m \quad \beta-a.s.
\end{equation*}
for some $\bar{c}_m \geq 0$. We therefore get the claim with $c_m=\bar{c}_m$, provided
we show that $\bar{c}_m>0$.

In view of \eqref{caprepr*}, we can get a lower bound on capacity by choosing a test
probability measure. We again choose the sojourn measure of $W_1[0,t]$ in \eqref{muWiener}.
This gives
\begin{equation*}
\frac{t}{\cp(W_1[0,t])} \leq t \int_{\R^m} \int_{\R^m}
\frac{\mu_{W_1[0,t]}(dx) \mu_{W_1[0,t]}(dy)}{\kappa_m|x-y|^{m-2}}
= \frac{1}{t} \int_0^t du \int_0^t dv \int_{\R^m} \int_{\R^m}
\frac{\nu_{\beta(u)}(dx) \nu_{\beta(v)}(dy)}{\kappa_m|x-y|^{m-2}}.
\end{equation*}
Now, there exists a $C<\infty$ such that
\begin{equation*}
\int_{\R^m} \int_{\R^m}
\frac{\nu_a(dx) \nu_b(dy)}{\kappa_m|x-y|^{m-2}}
\leq \frac{C}{|a-b|^{m-2} \vee 1} \qquad \forall\,a,b \in \R^m.
\end{equation*}
Hence
\begin{equation*}
\frac{t}{\cp(W_1[0,t])} \leq \frac{1}{t} \int_0^t du \int_0^t dv \,\,
\frac{C}{|\beta(u)-\beta(v)|^{m-2} \vee 1}.
\end{equation*}
To prove that $\bar{c}_m>0$ it suffices to show that the right-hand side has a finite
expectation. To that end, we estimate
\begin{equation*}
\frac{1}{t} \int_0^t du \int_0^t dv \,\,
\E_0\left(\frac{1}{|\beta(u)-\beta(v)|^{m-2} \vee 1}\right)
\leq 2 \int_0^t dv \,\,
\E_0\left(\frac{1}{|\beta(v)|^{m-2} \vee 1}\right),
\end{equation*}
and note that, as shown in \eqref{vintest}, the integral converges as $t\to\infty$ when
$m \geq 5$.

%%%%%%%%%%%%%%%%% REFERENCES %%%%%%%%%%%%%%%%%%

%%%

\end{document}